%% file: main.tex
\documentclass[11pt,twoside]{amsart}
\title{Permutation ensembles on products of simplices}
 \author{Suho Oh}
 \address{Texas State University}
 \email{suhooh@txstate.edu}
\usepackage{prefix} 
\date{}
\begin{document}
    \begin{abstract}
We propose the study of $S_n$-ensembles: $n \times n$ arrays of permutations of $[n]$ that encode the boundary data of $n\Delta_{n-1}$. We develop a general toolkit for these ensembles, valid for all $n$, and use it to settle the case $n = 4$. Our main result characterizes which boundary data extend to the interior: an $S_4$-ensemble contains a permutation appearing four times if and only if it avoids a single forbidden configuration, which we call a \emph{pattern}. We show moreover that the pattern is rigid (an $S_4$-ensemble containing one is unique up to symmetry), so that the non-extendable boundaries form a single orbit.
    \end{abstract}
    \maketitle


\section{Introduction}

We begin with a puzzle. In a $3$-by-$3$ table $\T$, place a permutation of $[3] = \{1,2,3\}$ in each cell (written in one-line notation, so $231$ sends $1 \mapsto 2$, $2 \mapsto 3$, $3 \mapsto 1$), subject to one rule: the entry $\T_{i,j}$, in row $i$ and column $j$, must have the number $i$ in its $j$-th position. We underline these forced entries: \cref{tab:s3ensemble} satisfies the rule.
\begin{table}[h]
\centering
$$\begin{matrix}
\centering
\uu{1}32 & 3\uu{1}2 & 23\uu{1}\\
\uu{2}31 & 3\uu{2}1 & 31\uu{2}\\
\uu{3}21 & 2\uu{3}1 & 21\uu{3}
\end{matrix}$$
\caption{An $S_3$-ensemble.}
\label{tab:s3ensemble}
\end{table}

When the underline position is clear from where the entry sits, we may omit it, so that repeated values can be read off at a glance: in \cref{tab:s3ensemble}, $\T_{2,1} = \T_{3,2} = \T_{1,3} = 231$.

Two entries $\T_{i,j}$ and $\T_{i',j'}$ are \newword{linked}, written $\T_{i,j} \sim \T_{i',j'}$, if their permutations (ignoring the underlines) differ by a single transposition: in the table above, $\uu{1}32 \sim 3\uu{1}2$, but $\uu{1}32 \not\sim \uu{3}21$. The table satisfies \newword{linkage} if every entry $\T_{i,j}$ is linked to some entry $\T_{i,j'}$ of its row and some entry $\T_{i',j}$ of its column. A $3 \times 3$ table of permutations satisfying the forced-entry and linkage conditions is an \newword{$S_3$-ensemble}.

It turns out that $S_3$-ensembles are related to fine mixed subdivisions of $3\Delta_2$ (which come from triangulations of $\Delta_2 \times \Delta_2$): an $S_3$-ensemble specifies the boundary (the numbers shown on the exterior in \cref{fig:3d2trivoro}), and the fact that $231$ appears three times in \cref{tab:s3ensemble} corresponds to the fact that this boundary can be completed to the fine mixed subdivision drawn in \cref{fig:3d2tri} (for more details, see \cref{ex:readedge} and \cref{ex:readtable}). In fact, the assertion that every such table contains a permutation appearing three times is equivalent to the acyclic system theorem of Ardila and Ceballos \cite{ArdilaCeballos13} for $3\Delta_2$.

\begin{figure}[ht]
\centering
\begin{subfigure}[t]{0.45\textwidth}
  \centering
  \resizebox{\linewidth}{!}{\input{tikz/3d2trilb}}
  \caption{Labelled triangulation.}
  \label{fig:3d2trilb}
\end{subfigure}
\hfill
\begin{subfigure}[t]{0.45\textwidth}
  \centering
  \resizebox{\linewidth}{!}{\input{tikz/3d2trivoro}}
  \caption{Its Voronoi dual.}
  \label{fig:3d2trivoro}
\end{subfigure}
\caption{A fine mixed subdivision of $3\Delta_2$ corresponding to \cref{tab:s3ensemble}.}
\label{fig:3d2tri}
\end{figure}

From Theorem 4.2 (2-D acyclic system theorem) of \cite{ArdilaCeballos13}, which describes exactly what kinds of exterior labeling can be extended to a fine mixed subdivision of $n\Delta_2$, we get the following as a corollary.

\newtheorem*{thm:s3ensemble}{Theorem \ref{thm:s3ensemble}}
\begin{thm:s3ensemble}
Any $S_3$-ensemble has some $w \in S_3$ that appears $3$ times.
\end{thm:s3ensemble}

Our goal is to answer a similar question in higher dimensions: when the boundary information is given, exactly when can we complete the fine mixed subdivision of the interior of $n\Delta_{d-1}$? Throughout the paper we focus on the ``square'' case $n\Delta_{n-1}$, that is, $d = n$. We discuss the obstacles in the general case $n\Delta_{d-1}$ in \cref{rem:whyacychard} and \cref{sec:further}.

A permutation $w = w_1 \dots w_n \in S_n$ with one underlined entry $\uu{w_i}$ can be read as a matching between the equicardinal sets $[n] \setminus i$ and $[n] \setminus w_i$, obtained by removing the underlined position and value. Any $\{i_1, \dots, i_k\} \subseteq [n] \setminus i$ then induces the submatching $i_j \mapsto w_{i_j}$. For an $n \times n$ table satisfying the forced-entry condition, an \newword{induced submatching} between equicardinal $I, J \subsetneq [n]$ is any submatching between $I$ and $J$ induced by an entry of the table. For example, in \cref{tab:s3ensemble}, the induced submatching between positions $23$ and entries $13$ only consists of $\{23,31\}$: the matching sending position $2$ to entry $3$ and position $3$ to entry $1$.

\begin{definition}
\label{def:snensemble}
An $S_n$-ensemble is an $n$-by-$n$ table $\T$ where each entry is an element of $S_n$ with the following properties:
\begin{itemize}
    \item (forced-entry) If we write $\T_{i,j} = w_1 \dots \uu{w_j} \dots w_n$, then $w_j = i$.
    \item (linkage) For each $i,j \in [n]$, there is some $i', j' \in [n]$ such that $\T_{i,j} \sim \T_{i,j'}$ and $\T_{i,j} \sim \T_{i',j}$.
    \item (compatible) For each equicardinal $I,J \subsetneq [n]$, there is at most one induced submatching between $I$ and $J$.
\end{itemize}
\end{definition}

For the compatible condition, when $|I| = |J| = k$ we will call it the \newword{$k$-compatible} condition. For example, we will not be able to see $\uu{1}432$ and $1\uu{4}23$ in the same $S_4$-ensemble, since they are not 2-compatible. On the other hand, $\uu{1}432$ and $4\uu{1}32$ are 2-compatible.

\begin{remark}
\label{rem:s3compat}
For $n = 3$, the compatible condition holds automatically, which is why it was omitted from the informal definition of $S_3$-ensembles above. Indeed, $1$-compatibility is vacuous, as there is only one possible matching between two singletons, and an induced submatching between $2$-subsets $I = [3] \setminus \{j\}$ and $J = [3] \setminus \{i\}$ can only come from the single entry $\T_{i,j}$, so uniqueness is automatic.
\end{remark}

We call the cells $(w(j), j)$, $j \in [n]$, the \newword{table cells} of $w \in S_n$. These are exactly the cells at which $w$ may appear as an entry, so $w$ appears at most $n$ times in $\T$, at most once in each row and column. In \cref{thm:boundary} we make the connection between $S_n$-ensembles and fine mixed subdivisions precise: the boundary data of a fine mixed subdivision of $n\Delta_{n-1}$ gives rise to an $S_n$-ensemble, and the boundary data can be completed to a fine mixed subdivision of $n\Delta_{n-1}$ if and only if some $w \in S_n$ appears $n$ times in the associated table. This motivates the following question.

\begin{question}
\label{ques:snensemble}
Can one describe the exact condition for an $S_n$-ensemble to have some $w \in S_n$ that appears $n$ times?
\end{question}

The main results of this paper answer \cref{ques:snensemble} for $n = 4$ completely, and for $n = 5$ completely modulo computer verification. We stress the scale of the question: for $n = 3$ it recovers the acyclic system theorem of Ardila and Ceballos \cite{ArdilaCeballos13}, via the identification of $S_3$-ensembles with acyclic systems on $3\Delta_2$ (\cref{thm:s3ensemble}). Moreover, the $n = 3$ story governs the whole two-dimensional theorem: acyclicity of a system of permutations on $n\Delta_2$ is detected on triples of labels, and the restriction to each triple is exactly an $S_3$-ensemble, so the acyclic system theorem for $n\Delta_2$ is a local-to-global statement whose local model is $3\Delta_2$. The case $n = 4$ settled here is thus the first instance beyond a result that was itself the main theorem of a prior paper. For $n = 4$:

\newtheorem*{thm:s4ensemble}{Theorem \ref{thm:s4ensemble}}
\begin{thm:s4ensemble}
An $S_4$-ensemble has some $w \in S_4$ that appears $4$ times if and only if it is \emph{triangle-free}: its blocking graph contains no directed triangle, i.e.\ no \emph{pattern}.
\end{thm:s4ensemble}

Moreover, the obstruction is rigid:

\newtheorem*{thm:rigid}{Theorem \ref{thm:rigid}}
\begin{thm:rigid}
An $S_4$-ensemble containing a pattern is unique up to symmetry: it is the explicit ensemble $E^*_4$ of \cref{tab:estar4}. Consequently the patterns and the $S_4$-ensembles without a permutation of multiplicity four are in bijection, and there are exactly $192$ of each.
\end{thm:rigid}

Both theorems are proved through a directed graph that we attach to every ensemble in \cref{sec:algebra}. The \newword{blocking graph} of $\T$ has vertex set $S_n$, with an edge $u \to_j v$ whenever the table cell of $u$ in column $j$ carries the entry $v \neq u$: the permutation $v$ \emph{blocks} $u$ from appearing there. The out-degree of $u$ is $n - \mathrm{mult}(u)$, so the permutations of multiplicity $n$ are exactly the sinks. A graph with no sink has a directed cycle, and there are no cycles of length two, so the shortest possible cycle is a triangle. For $n = 4$ we take this directed triangle as our central object, calling it a \newword{pattern}. It is forced into a single normal form (\cref{lem:patternform}). In this language, \cref{thm:s4ensemble} is a girth statement: the blocking graph of an $S_4$-ensemble has a sink if and only if it contains no directed triangle.

The forward direction splits into two independent halves. 

\newtheorem*{thm:mult3}{Theorem \ref{thm:mult3}}
\begin{thm:mult3}
Every $S_4$-ensemble contains a permutation that appears at least $3$ times.
\end{thm:mult3}

The second half is a promotion lemma (\cref{lem:mult3done}): in a triangle-free ensemble, multiplicity three promotes to multiplicity four. Together they prove the girth statement. We emphasize that most of the work is done by tools valid for every $n$: the algebraic dictionary of \cref{sec:algebra} turns linkage and compatibility into one-sided multiplication and cycle conditions on quotients, and yields blocker tools (\cref{subsec:blocker}) governing permutations of near-full multiplicity. At $n = 4$ the dictionary leaves so little room that the link graph of every line is forced to be a spanning tree: a star or a path (\cref{cor:linetree}). The $n = 4$ proofs assemble these lemmas, with the special role of $n = 4$ isolated at a handful of counting coincidences.

For $n = 5$ the mechanism persists. In \cref{sec:further} we exhibit an $S_5$-ensemble $E^*_5$ in which no permutation appears five times, show by computer that it is the unique such ensemble up to symmetry, and that it arises from a genuine fine mixed subdivision structure on the boundary of $5\Delta_4$. The mechanism behind both $E^*_4$ and $E^*_5$ is the same: a directed triangle in the blocking graph among permutations of near-full multiplicity, each blocking the next from reaching full multiplicity. In this sense $E^*_5$ is the natural extension of $E^*_4$, and we propose such \newword{blocking cycles} as the organizing principle for general $n$.

The paper is organized as follows. \cref{sec:prod} reviews products of simplices and matching ensembles, and proves the correspondence between $S_n$-ensembles and boundary data of $n\Delta_{n-1}$. \cref{sec:algebra} develops an algebraic dictionary for $S_n$-ensembles that makes every subsequent verification a short computation with permutations: linkage becomes one-sided multiplication, compatibility becomes a cycle condition on quotients (\cref{lem:compat}), and the symmetry group acts transitively on cells (\cref{lem:symmetry}), all valid for every $n$, with the $n = 4$ specialization forcing the link graph of every line to be a star or a path (\cref{cor:linetree}). The section then introduces the blocking graph with its basic invariants (\cref{prop:blockdeg,lem:girth}), and closes with the blocker tools of \cref{subsec:blocker}, the general-$n$ lemmas on which the $n = 4$ argument rests. \cref{sec:main} proves the three theorems above, via the multiplicity theorem (\cref{thm:mult3}) and the promotion lemma (\cref{lem:mult3done}). \cref{sec:further} exhibits the unique $n = 5$ obstruction $E^*_5$ and collects further directions, including the general conjectures the new framework isolates.

\section{Connections to products of simplices}
\label{sec:prod}

\subsection{Basics on products of simplices}

The product of two simplices is a fundamental polytope whose triangulations exhibit rich combinatorial structure. We refer to \cite{Santos05} for background, and to \cite{zbMATH05192027}, \cite{ArdilaCeballos13} for the connections to Schubert calculus and tropical geometry that motivate this paper. Under the Cayley trick of Santos \cite{zbMATH02223039}, these triangulations become the fine mixed subdivisions of $n\Delta_{d-1}$ that we study, and the same data reappears through matching ensembles and through the trianguloids of Galashin, Nenashev, and Postnikov, which encode triangulations of root polytopes \cite{GalashinNenashevPostnikov23} in the spirit of \cite{KalmanPostnikov17}. Even the global structure of these triangulations is subtle: the flip graph of $\Delta_3 \times \Delta_{d-1}$ is connected \cite{Liu20}, but that of $\Delta_4 \times \Delta_{d-1}$ is not for large $d$ \cite{Liu18}. Closest to our concerns is the tropical and oriented-matroid strand: polyhedral matching fields, whose linkage-based input is close in spirit to our ensembles, produce oriented matroids (not necessarily realizable) from such triangulations \cite{zbMATH07843943}, with a topological representation theorem for tropical oriented matroids in \cite{zbMATH06535105}. Most directly, the extendability results of Ceballos, Padrol, and Sarmiento for partial triangulations of $\Delta_{n-1}\times\Delta_{n-1}$, obtained via their Dyck path triangulation \cite{zbMATH06381976} (extended abstract \cite{zbMATH06577215}), are the closest precedent for the completability question we study on the boundary of $n\Delta_{n-1}$. Further connections include non-regular triangulations \cite{zbMATH00881214}, cellular resolutions \cite{DOCHTERMANN2012304}, subdivision algebras \cite{zbMATH07814959}, face structure \cite{zbMATH07419599}, and size bounds for triangulations of simplotopes \cite{zbMATH06825996}.

We denote the simplex formed by the convex hull of $e_1,\ldots,e_n$ as $\Delta_{n-1}$, and write $n\Delta_{d-1}$ for the Minkowski sum $\Delta_{d-1} + \cdots + \Delta_{d-1}$ ($n$ summands). The vertices of $\Delta_{d-1}$ are labeled $1', \dots, d'$.

\begin{definition}
\label{def:fms}
A \newword{mixed cell} of $n\Delta_{d-1}$ is a Minkowski sum $B_1 + \cdots + B_n$, where each $B_i$ is a face of $\Delta_{d-1}$. A \newword{mixed subdivision} of $n\Delta_{d-1}$ is a collection of mixed cells, together with their Minkowski decompositions, whose union is $n\Delta_{d-1}$ and such that any two cells intersect in a common face, compatibly with the chosen decompositions. A mixed subdivision is \newword{fine} if $\sum_i \dim B_i = d-1$ for every cell.
\end{definition}

\begin{remark}
All mixed subdivisions considered in this paper are fine. When there is no risk of confusion we abbreviate ``fine mixed subdivision'' to \newword{subdivision}.
\end{remark}

For $d = 3$, a fine mixed subdivision of $n\Delta_2$ is a subdivision of an equilateral triangle of side length $n$ into upward unit triangles and unit rhombi, as illustrated in \cref{fig:3d2trilb}. These subdivisions, also known as lozenge tilings \cite{Santos05}, are of particular interest due to their connections with matroids, tropical hyperplane arrangements, Schubert calculus \cite{zbMATH05192027}, and tropical oriented matroids \cite{ArdilaDevelin07}, \cite{OYoo10}, \cite{Horn16}. Ardila and Billey demonstrated that the configuration of the unit triangles must be spread out: for each $k$, no size $k$ subtriangle can contain more than $k$ unit triangles.

\begin{theorem}[Spread out simplices \cite{zbMATH05192027}]
\label{thm:2dspread}
A collection of $n$ triangles in $n\Delta_2$ can be extended to a fine mixed subdivision if and only if it is spread out.
\end{theorem}

Given a fine mixed subdivision of $n\Delta_{d-1}$ and a boundary edge of $n\Delta_{d-1}$ lying over the edge $x'y'$ of $\Delta_{d-1}$, the cells of the subdivision restrict to a subdivision of that edge into $n$ unit segments, and each unit segment is a translate of the unique summand $B_i$ of its cell that contains the edge $x'y'$. Since the subdivision is fine, each index $i \in [n]$ contributes exactly one unit segment, so reading off the contributing indices as we move from $x'$ to $y'$ yields a permutation of $[n]$. The \newword{system of permutations} of the subdivision is the collection of these permutations over all edges of $\Delta_{d-1}$. A system of permutations is called \newword{acyclic} if, when the three permutations on the edges of any triangle $x'y'z'$ are read clockwise starting from a vertex, no cycle of the form $\ldots i \ldots j \ldots, \ldots i \ldots j \ldots, \ldots i \ldots j \ldots$ is seen.

\begin{theorem}[2-D acyclic system \cite{ArdilaCeballos13}]
\label{thm:2dacyc}
A system of permutations comes from a fine mixed subdivision of $n\Delta_2$ if and only if it is acyclic.
\end{theorem}

In \cref{fig:3d2trivoro}, if we start from the vertex $1'$ and move clockwise, we get $123,213,321$, which is acyclic. However, if we swap the entries $1$ and $3$ on the edge $1'3'$ to get $123,213,123$, this is not acyclic, as the pattern $13,13,13$ appears as we traverse clockwise.

A \newword{triangulation} of a polytope is a polyhedral decomposition of the polytope into simplices intersecting face-to-face. The \newword{Cayley trick} gives a bijection between the triangulations of $\Delta_{n-1} \times \Delta_{d-1}$ and the fine mixed subdivisions of $n\Delta_{d-1}$, moreover, each simplex inside such a triangulation can be encoded by a spanning tree of the complete bipartite graph $K_{n,d}$: label the left vertices with $[n]$ and the right vertices with $[d]$, and encode the vertex $(e_i,e'_j)$ of the simplex as the edge $(i,j)$. See \cref{fig:cayley7}.

\begin{figure}
\centering
\input{tikz/cayley_tikz}
\caption{How the triangulation relates to the fine mixed subdivisions of $2\Delta_2$.}
\label{fig:cayley7}
\end{figure}

\begin{conjecture}[Spread out simplices \cite{zbMATH05192027}]
\label{conj:spread}
A collection of $n$ unit simplices in $n\Delta_{d-1}$ can be extended to a fine mixed subdivision if and only if it is spread out.
\end{conjecture}

\begin{conjecture}[Acyclic system \cite{ArdilaCeballos13}]
\label{conj:acyc}
A system of permutations comes from a fine mixed subdivision of $n\Delta_{d-1}$ if and only if it is acyclic.
\end{conjecture}

\cref{conj:spread}, the higher-dimensional extension of \cref{thm:2dspread}, is still wide open, whereas a counterexample for \cref{conj:acyc} was found in \cite{Santos13} for $n=5$ and $d=4$. A minimal counterexample will be provided in \cref{prop:acyccounter}, using the tools that we introduce next.

\begin{remark}
\label{rem:whyacychard}
One insight for why \cref{conj:acyc} is so hard is that there is too big a dimensional gap between what is given and what we have to fill: in \cref{thm:2dacyc} the $1$-dimensional boundary data fills a $2$-dimensional interior, but in \cref{conj:acyc} the same $1$-dimensional data must fill interiors of dimension $2, \dots, d-1$.
\end{remark}

\subsection{Matching ensembles}
The main tool we use is a collection of matchings that encodes the fine mixed subdivision of $n\Delta_{d-1}$, called \newword{matching ensembles} \cite{OYoo15}, \cite{zbMATH07399067} or \newword{matching fields with linkage} \cite{LohoSmith20}. Matching ensembles are one of several matroid-like encodings of these subdivisions, alongside tope arrangements \cite{LohoSmith20} and trianguloids \cite{GalashinNenashevPostnikov23}, all governed by variants of the linkage axiom, and Yao \cite{Yao25} gives a unified account of their cryptomorphisms.

\begin{figure}
\centering
\input{tikz/33mixed_tikz}
\input{tikz/matchings_tikz}
\caption{Spanning trees describing the cells of a fine mixed subdivision of $3\Delta_2$, and the collection of matchings obtained from them.}
\label{fig:sptrees}
\end{figure}

Collecting, over all spanning trees of a fixed fine mixed subdivision, the matchings that appear as a subgraph of at least one tree yields a collection of matchings of size $\leq \min(n,d)$ between the left vertex set $[n]$ and the right vertex set $[d]$, with the following three properties.

\begin{definition}
\label{def:ME}
Fix $n,d$ and look at the matchings within $K_{n,d}$. We say that a collection of matchings between all $A \subset [n], B \subset [d]$ with $|A| = |B|$ forms a matching ensemble if the following properties are satisfied.
\begin{enumerate}
\item For fixed $A$ and $B$, the matching between them is unique.
\item The submatching of a matching must also be within the collection.
\item (linkage) For any matching $M$ in the collection and any vertex $v$ outside its left vertex set (left linkage) or outside its right vertex set (right linkage), the collection contains a matching obtained from $M$ by replacing a single edge $(i,j)$ of $M$ with the edge $(v,j)$ (left linkage) or with the edge $(i,v)$ (right linkage).
\end{enumerate}
\end{definition}

For example, in \cref{fig:sptrees} we can collect $(1,2'),(2,1')$ from the top rhombus, while the matching $(1,1'),(2,2')$ never appears. Starting from $(1,2'),(2,1')$ and picking the vertex $3$ on the left, we find the matching $(1,2'),(3,1')$, as the linkage property guarantees.

\begin{remark}
\label{rem:stronglinkage}
The linkage condition in \cref{def:ME}(3) is the weak, single-exchange form. For matching fields, Bernstein and Zelevinsky showed it to be equivalent to the strong form in which, for every $A \subseteq [n]$ and $B \subseteq [d]$ with $|A| = |B| + 1$, the union of the matchings between $A \setminus \{a\}$ and $B$ over all $a \in A$ is a spanning tree whose $B$-side degrees all equal $2$ (the \newword{linkage covector}), and dually \cite[Theorem 2]{BernsteinZelevinsky93}, see \cite[Theorem 2.28]{Yao25} for five equivalent formulations. Yao defines matching ensembles directly by these covector axioms \cite[Definition 2.33]{Yao25}, and the cryptomorphism with pointed matching fields \cite[Theorem 2.34]{Yao25} shows the two axiom systems define the same objects, so the linkage condition of \cref{def:ME} may be replaced by the covector form at no cost (the submatching and linkage conditions can even be merged into a single extended covector axiom, \cite[Remark 2.35]{Yao25}). We caution that this equivalence uses matchings of all sizes together with the submatching condition. An $S_n$-ensemble carries only the maximal layer (\cref{rem:loosedef}), so for our tables the strong form is a genuine strengthening in general. At $n = 4$ it is recovered from the ensemble axioms by the Line Dichotomy (\cref{cor:linetree}), and for tables of boundary matching ensembles it is automatic (\cref{rem:linetreegeom}).
\end{remark}

\begin{theorem}[{\cite{OYoo15}, \cite[Theorem 2.36]{Yao25}}]
\label{thm:MEsubdiv}
There is a bijection between the matching ensembles of $[n]$ and $[d]$ and the triangulations of $\Delta_{n-1} \times \Delta_{d-1}$.
\end{theorem}

\begin{remark}
\label{rem:restriction}
The restriction of a fine mixed subdivision of $n\Delta_{d-1}$ to the face spanned by the vertices $x_1', \dots, x_k'$ of $\Delta_{d-1}$ is a fine mixed subdivision of $n\Delta_{k-1}$, and under \cref{thm:MEsubdiv} its matching ensemble is the restriction of the original one: the matchings between subsets of $[n]$ and subsets of $\{x_1, \dots, x_k\}$.
In particular, the $k$-by-$k$ matchings encode what happens on the boundary faces of dimension $k-1$, and the $2$-by-$2$ matchings encode what happens at the edges.
\end{remark}

The interpretation of a $2$-by-$2$ matching $\{ax,by\}$ is: at the boundary edge $x'y'$ of $n\Delta_{d-1}$, reading the labels from $x'$ to $y'$, we get $\dots b \dots a \dots$.

\begin{example}
\label{ex:readedge}
Take the boundary information given in \cref{fig:3d2trivoro}. Within the edge $1'3'$, consider the edge labels $2,3$. As we move from $1'$ to $3'$, the labels are read in order $2,3$. So, the matching we get from this is $\{23',31'\}$.
\end{example}

\begin{proposition}
\label{prop:edgeinterpret}
Fix $n$ and $d$. The construction of \cref{ex:readedge} is a bijection between systems of permutations on $n\Delta_{d-1}$ and collections consisting of exactly one $2$-by-$2$ matching between each pair of a $2$-subset $\{a,b\} \subseteq [n]$ and an edge $\{x,y\}$ of $\Delta_{d-1}$, satisfying left linkage. Moreover, under this bijection, the system of permutations is acyclic if and only if the collection also satisfies right linkage.
\end{proposition}
\begin{proof}
Given a system of permutations, for each edge $\{x,y\}$ and each pair $\{a,b\} \subseteq [n]$ we record the relative order of $a$ and $b$ in the permutation on that edge as a $2$-by-$2$ matching, obtaining exactly one matching for each such pair.

We first show that such a collection satisfies left linkage if and only if, for each edge $\{x,y\}$, the matchings involving that edge come from a total order on $[n]$. Fix the edge and consider the tournament on $[n]$ where $a \to b$ whenever $a$ is read before $b$. Suppose the tournament has a $3$-cycle, then the collection contains matchings of the form $\{ay,bx\}, \{by,cx\}, \{cy,ax\}$. Taking $M = \{ay,bx\}$ and $v = c$, left linkage would require the collection to contain $\{ay,cx\}$ or $\{cy,bx\}$, but it contains $\{cy,ax\}$ and $\{by,cx\}$ between those pairs, so left linkage fails. Conversely, if every such tournament is transitive, take any $M = \{ay,bx\}$ and $v = c$. Comparability of $c$ with $a$ and $b$ in the total order guarantees that one of the two replacement matchings is the recorded one.

Next, right linkage. If the system is not acyclic, then for some triangle $x'y'z'$ and labels $a,b$ the pattern $ababab$ appears, corresponding to matchings $\{ay,bx\}$, $\{ax,bz\}$, $\{az,by\}$. Taking $M = \{ay,bx\}$ and $v = z$, right linkage would require $\{ay,bz\}$ or $\{az,bx\}$, but the collection contains the other matchings between those pairs, so right linkage fails. Conversely, if right linkage fails at some $M = \{ay,bx\}$ and $z$, then the matchings between $\{a,b\}$ and $\{x,z\}$, $\{y,z\}$ are forced to be $\{ax,bz\}$ and $\{az,by\}$, which together with $M$ form the cyclic pattern.
\end{proof}

\begin{remark}
\label{rem:edgeinterpretlit}
\cref{prop:edgeinterpret} can also be assembled from the matching-ensemble literature. For a single edge, the closure and right linkage axioms of an $(n,2)$-matching ensemble are automatically satisfied (the singleton matchings are forced), so such an ensemble is exactly a collection of $2$-by-$2$ matchings satisfying left linkage, while fine mixed subdivisions of $n\Delta_1$ are orderings of the $n$ unit segments. The first statement is thus the $d = 2$ case of the correspondence between matching ensembles and fine mixed subdivisions (\cite{OYoo15}, \cite[Theorem 2.36]{Yao25}), applied to each edge of $\Delta_{d-1}$. The second statement identifies acyclicity with the smallest case of the right linkage axiom (\cite[Definition 2.33]{Yao25}): for a pair $\{a,b\}$ and a triangle of $\Delta_{d-1}$, the union of the three $2$-by-$2$ matchings is a tree exactly when the cyclic pattern $ababab$ does not occur, the same condition that reappears as the hexagon axiom of trianguloids \cite{GalashinNenashevPostnikov23} (see \cite[Remark 4.5]{Yao25}). We keep the short self-contained proof above, to help the reader.
\end{remark}

From \cref{prop:edgeinterpret} we can translate \cref{thm:2dacyc} and \cref{conj:acyc}: the acyclic system theorem states that a collection of $2$-by-$2$ matchings between subsets of $[n]$ and $[3]$ satisfying the matching ensemble axioms can always be completed with $3$-by-$3$ matchings to a matching ensemble, and the (false) acyclic system conjecture claims the analogous statement for general $[d]$.

\subsection{Permutation ensembles and boundaries of $n\Delta_{n-1}$}

Given an $(n-1)$-by-$(n-1)$ matching between subsets of $[n]$, we can add the edge connecting the missing entries from each side and think of it as a permutation of $[n]$, with an underline encoding which edge was added. For example, the matching $\{12,34,41\}$ is encoded as $2\uu{3}41$. Under this encoding, the entry $\T_{i,j}$ of a table corresponds to the matching between the position set $[n] \setminus \{j\}$ and the value set $[n] \setminus \{i\}$, by \cref{rem:restriction}, row $i$ of the table records the maximal matchings of the fine mixed subdivision of the facet of $n\Delta_{n-1}$ opposite the vertex $i'$.

\begin{definition}
\label{def:boundaryME}
A \newword{boundary matching ensemble} of $n\Delta_{n-1}$ is a collection of matchings between all pairs $A \subsetneq [n]$, $B \subsetneq [n]$ with $|A| = |B| \geq 1$, satisfying conditions (1)--(3) of \cref{def:ME}. (In condition (3), the replacement matching again has right vertex set of size at most $n-1$, so it stays within the collection's range.)
\end{definition}

By \cref{thm:MEsubdiv} and \cref{rem:restriction}, restricting a boundary matching ensemble to the matchings with right vertex set inside $[n] \setminus \{i\}$ recovers the matching ensemble of a fine mixed subdivision of the facet opposite $i'$, and these facet subdivisions agree on common faces. Conversely, a choice of fine mixed subdivisions of the $n$ facets agreeing on common faces yields a boundary matching ensemble precisely when right linkage holds among the maximal $(n-1)$-by-$(n-1)$ matchings, which is the only condition of \cref{def:ME} involving more than one facet at a time.

\begin{theorem}[Boundary correspondence]
\label{thm:boundary}
Let $E$ be a boundary matching ensemble of $n\Delta_{n-1}$, and let $\T$ be the $n$-by-$n$ table whose entry $\T_{i,j}$ is the underlined permutation encoding the unique matching of $E$ between $[n]\setminus\{j\}$ and $[n]\setminus\{i\}$. Then:
\begin{enumerate}
\item[(a)] $\T$ is an $S_n$-ensemble.
\item[(b)] $E$ extends to a matching ensemble between $[n]$ and $[n]$ (equivalently, by \cref{thm:MEsubdiv} and \cref{rem:restriction}, the fine mixed subdivision of the boundary of $n\Delta_{n-1}$ encoded by $E$ extends to a fine mixed subdivision of $n\Delta_{n-1}$) if and only if some $w \in S_n$ appears $n$ times in $\T$.
\end{enumerate}
\end{theorem}
\begin{proof}
(a) The forced-entry condition is immediate. For linkage, let $M$ be the matching corresponding to $\T_{i,j} = w$, so $M = \{(p, w_p) : p \neq j\}$ and $w_j = i$. Applying left linkage of $E$ to $M$ with $v = j$ produces $M' = M - (j', w_{j'}) + (j, w_{j'})$ for some $j' \neq j$, as underlined permutations, $M'$ is obtained from $w$ by swapping the entries at positions $j$ and $j'$ and moving the underline to $j'$, so $M'$ is the entry $\T_{i,j'}$ and $\T_{i,j} \sim \T_{i,j'}$. Right linkage with $v = i$ similarly yields a linked $\T_{i',j}$. For compatibility, if two entries induced distinct submatchings between the same $(I,J)$ with $J \subsetneq [n]$, then by condition (2) both would belong to $E$, contradicting condition (1).

(b) If $\T_{w_j, j} = w$ for all $j$, let $M_w = \{(p,w_p)\}$ and $E' = E \cup \{M_w\}$. Uniqueness holds since the only new pair is $([n],[n])$. Every proper submatching of $M_w$ is a submatching of $M_w - (j,w_j)$, the matching of $\T_{w_j,j}$, hence lies in $E$, and linkage is vacuous for $M_w$ and unchanged for $E$. By \cref{thm:MEsubdiv}, $E'$ encodes a fine mixed subdivision of $n\Delta_{n-1}$ with boundary data $E$. Conversely, an extension $E'$ contains a unique perfect matching $M_w$, its $(n-1)$-submatchings lie in $E$, and by uniqueness they are the entries $\T_{w_j,j}$, so $w$ appears $n$ times.
\end{proof}

\begin{example}
\label{ex:readtable}
Take the boundary information given in \cref{fig:3d2trivoro}. In \cref{ex:readedge}, we obtained the matching $\{23',31'\}$, encoded as $\uu{2}31$: the entry in the second row and first column of \cref{tab:s3ensemble}. All other entries are obtained similarly.
\end{example}

\begin{example}
\label{ex:staircase4}
Let $n = 4$ and place the permutation $1234$ on every edge of $\Delta_3$, read from the vertex with the smaller label. \cref{fig:staircase4} displays a fine mixed subdivision of the boundary of $4\Delta_3$ inducing this data: on the facet with vertex set $x < y < z$, the four unit triangles line up along the edge $x'z'$ in the order $1, 2, 3, 4$ starting from $x'$, and six rhombi fill the rest. Reading off the maximal matchings as in \cref{ex:readtable} yields the table of \cref{tab:staircase4}. For instance, $\T_{1,4}$ is the matching of the summands $1, 2, 3$ into the facet opposite $1'$. On the edge $3'4'$ the pair $\{1,2\}$ is read in the order $1, 2$ starting from $3'$, so this matching sends $1$ to $4'$ and $2$ to $3'$, and comparing with the summand $3$ on the remaining edges gives $3 \mapsto 2'$, hence $\T_{1,4} = 432\uu{1}$. The permutation $4321$ appears four times in the table, at the antidiagonal cells, so \cref{thm:boundary} confirms that this boundary extends to a fine mixed subdivision of $4\Delta_3$: following the proof of the theorem, the extension adjoins the perfect matching sending each summand $a$ to the vertex $(5-a)'$.
\end{example}

\begin{figure}
\centering
\resizebox{\textwidth}{!}{\input{tikz/staircase4}}
\caption{A fine mixed subdivision of the boundary of $4\Delta_3$ inducing the permutation $1234$ on every edge, read from the vertex with the smaller label. The tetrahedron is viewed from above the vertex $4'$, so that the three facets containing $4'$ are seen from the outside, and the bottom facet $1'2'3'$ is unfolded along the edge $2'3'$. Each unit triangle is colored and labeled by its summand, and the small numbers along each edge record the induced permutation.}
\label{fig:staircase4}
\end{figure}

\begin{table}[h]
\centering
$$\begin{matrix}
\uu{1}432 & 4\uu{1}32 & 43\uu{1}2 & 432\uu{1}\\
\uu{2}431 & 4\uu{2}31 & 43\uu{2}1 & 431\uu{2}\\
\uu{3}421 & 4\uu{3}21 & 42\uu{3}1 & 421\uu{3}\\
\uu{4}321 & 3\uu{4}21 & 32\uu{4}1 & 321\uu{4}\\
\end{matrix}$$
\caption{The $S_4$-ensemble of the boundary subdivision in \cref{fig:staircase4}. The permutation $4321$ appears four times, at the antidiagonal cells.}
\label{tab:staircase4}
\end{table}

\begin{example}
\label{ex:regular4}
The boundary of \cref{ex:staircase4} is as uniform as possible, so we also give an example with no such symmetry. Keep $n = 4$ and lift the vertex $v'$ of the $p$-th summand of $\Delta_3 + \Delta_3 + \Delta_3 + \Delta_3$ to the height $h_p(v)$, where
$$\bigl(h_p(v)\bigr)_{p, v \in [4]} = \begin{pmatrix}0 & 4 & 1 & 0\\ 3 & 5 & 0 & 4\\ 3 & 6 & 6 & 0\\ 0 & 1 & 2 & 5\end{pmatrix}.$$
The lower faces of the lifted Minkowski sum project to a regular fine mixed subdivision of $4\Delta_3$, and \cref{fig:regular4} displays its restriction to the boundary. On the edge $x'y'$ with $x < y$ the induced permutation lists the summands in increasing order of the slope $h_p(y) - h_p(x)$, which gives $4231, 2143, 3124, 2134, 3124, 3142$ on the edges $1'2', 1'3', 1'4', 2'3', 2'4', 3'4'$. Reading off the maximal matchings as in \cref{ex:staircase4} yields the table of \cref{tab:regular4}. Note that no facet is tiled as in the staircase: on each facet the four unit triangles are scattered rather than lined up along an edge, and the four facet tilings are pairwise distinct. The permutation $1342$ appears four times in the table, at the cells $(1,1), (3,2), (4,3), (2,4)$, so \cref{thm:boundary} confirms that this boundary extends to a fine mixed subdivision of $4\Delta_3$. Here the extension is the regular subdivision we started with, and following the proof of the theorem it adjoins the unique perfect matching of that subdivision, which sends the summands $1, 2, 3, 4$ to the vertices $1', 3', 4', 2'$.
\end{example}

\begin{figure}
\centering
\resizebox{\textwidth}{!}{\input{tikz/regular4}}
\caption{The regular fine mixed subdivision of the boundary of $4\Delta_3$ induced by the lifting heights of \cref{ex:regular4}, drawn as in \cref{fig:staircase4}.}
\label{fig:regular4}
\end{figure}

\begin{table}[h]
\centering
$$\begin{matrix}
\uu{1}342 & 3\uu{1}42 & 43\uu{1}2 & 234\uu{1}\\
\uu{2}341 & 3\uu{2}41 & 43\uu{2}1 & 134\uu{2}\\
\uu{3}142 & 1\uu{3}42 & 41\uu{3}2 & 124\uu{3}\\
\uu{4}312 & 3\uu{4}12 & 13\uu{4}2 & 132\uu{4}\\
\end{matrix}$$
\caption{The $S_4$-ensemble of the boundary subdivision in \cref{fig:regular4}. The permutation $1342$ appears four times, at the cells $(1,1)$, $(3,2)$, $(4,3)$, $(2,4)$.}
\label{tab:regular4}
\end{table}

\begin{remark}
\label{rem:loosedef}
For $n \geq 4$, not every $S_n$-ensemble need correspond to the table of a boundary matching ensemble: our definition only imposes linkage between the size-$(n-1)$ matchings, while a boundary matching ensemble requires linkage at all sizes. \cref{thm:boundary} describes an extraction map, not a bijection.

In the framework of matching ensembles, an $S_n$-ensemble carries only the top layer of the structure, so extending it to a genuine boundary is a question about recovering the lower-size linkage. This gap is not vacuous: the ensemble $E^*_5$ of \cref{sec:further} does arise from a boundary matching ensemble, but whether every $S_4$-ensemble does is left open.
\end{remark}

Any $S_3$-ensemble is exactly an acyclic system of permutations on $3\Delta_2$, moreover, since the cyclic pattern involves only two labels at a time, a system of permutations on $n\Delta_2$ is acyclic if and only if its restriction to each triple of labels $a, b, c \in [n]$ (reading only the letters $a, b, c$ in each edge permutation) is an $S_3$-ensemble. So from \cref{thm:2dacyc}, we get the following.

\begin{theorem}[Acyclic system theorem using $S_3$-ensembles]
\label{thm:s3ensemble}
Any $S_3$-ensemble has some $w \in S_3$ that appears $3$ times.
\end{theorem}

Using this language, we can easily come up with a minimal counterexample for \cref{conj:acyc}.

\begin{proposition}
\label{prop:acyccounter}
    The minimal counterexample for \cref{conj:acyc} occurs for $n=4$ and $d=4$, with the system of permutations $1234, 1234, 1423, 2413, 4213, 1423$ on the edges $1'2',1'3',1'4',2'3',2'4',3'4'$, respectively.
\end{proposition}
\begin{proof}
By the submatching condition of \cref{def:ME}, every $3$-by-$3$ matching in a matching ensemble must have all of its $2$-by-$2$ submatchings in the ensemble. Since the $2$-by-$2$ matchings are prescribed by the given system of permutations, the candidates for each entry are exactly the underlined permutations all of whose $2$-by-$2$ submatchings appear among the prescribed matchings. Carrying this out yields \cref{tab:acyccounter}, all entries are determined except $\T_{3,2}$, which can be $4\uu{3}21$ or $2\uu{3}14$. If $\T_{3,2} = 2\uu{3}14$, linkage fails on $4\uu{1}23$ (no link in its column), if $\T_{3,2} = 4\uu{3}21$, linkage fails on $21\uu{3}4$ (no link in its row). So there is no fine mixed subdivision of $4\Delta_3$ fitting the given boundary information.
\begin{table}[h]
\centering
$$\begin{matrix}
\uu{1}324 & 4\uu{1}23 & 23\uu{1}4 & 432\uu{1}\\
\uu{2}314 & 4\uu{2}31 & 43\uu{2}1 & 431\uu{2}\\
\uu{3}214 & 4\uu{3}21,2\uu{3}14 & 21\uu{3}4 & 241\uu{3}\\
\uu{4}321 & 3\uu{4}21 & 23\uu{4}1 & 231\uu{4}\\
\end{matrix}$$
\caption{The table from $1234, 1234, 1423, 2413, 4213, 1423$ on $4\Delta_3$ does not satisfy linkage.}
\label{tab:acyccounter}
\end{table}
\end{proof}

The counterexample is illustrated in \cref{fig:acyccounter}. The obstruction is genuinely three-dimensional: each facet of $4\Delta_3$ admits a fine mixed subdivision inducing the given edge permutations, as displayed, and the four facet subdivisions realize the choice $\T_{3,2} = 2\uu{3}14$ in \cref{tab:acyccounter}. The failure of linkage identified in the proof occurs on $4\uu{1}23$ within its column, and a column mixes matchings from different facets, so the displayed facet subdivisions cannot be glued to a subdivision of the interior. With the other choice $\T_{3,2} = 4\uu{3}21$ the failure already occurs within row $3$, meaning the facet opposite $3'$ itself admits no fine mixed subdivision compatible with that entry. We also point out that with the choice $\T_{3,2} = 2\uu{3}14$ the permutation $2314$ appears at four of its table cells in \cref{tab:acyccounter}. This does not contradict \cref{thm:boundary}: since linkage fails, the table is not an $S_4$-ensemble and does not arise from a boundary matching ensemble, so the multiplicity criterion does not apply.

\begin{figure}
\centering
\resizebox{\textwidth}{!}{\input{tikz/acyccounter4}}
\caption{The minimal counterexample of \cref{prop:acyccounter}, drawn as in \cref{fig:staircase4}. The small numbers along the edges record the system of permutations $1234, 1234, 1423, 2413, 4213, 1423$ on the edges $1'2', 1'3', 1'4', 2'3', 2'4', 3'4'$. Each facet of $4\Delta_3$ is filled by a fine mixed subdivision inducing this edge data, and yet no fine mixed subdivision of $4\Delta_3$ has this boundary.}
\label{fig:acyccounter}
\end{figure}

\section{The algebra of $S_n$-ensembles}
\label{sec:algebra}

In this section we translate the defining conditions of an $S_n$-ensemble into the language of one-sided multiplication in $S_n$. Throughout, permutations act as maps from positions to values, so the entry $\T_{i,j}$ is a permutation $g$ with $g(j) = i$, and the table cells of $w$ are the cells $(w(j), j)$. Composition is right-to-left. Every verification in \cref{sec:main} reduces, via this dictionary, to a short computation.

\subsection{Linkage is one-sided multiplication}

When $g = \T_{i,j} \sim \T_{i,j'} = g'$, then $g' = g\,(j\ j')$: within row $i$ the value $i$ must move from position $j$ to position $j'$, and right multiplication by $(j\ j')$ is the only transposition of positions that does this. When $g = \T_{i,j} \sim \T_{i',j} = g'$, then $g' = (i\ i')\,g$: the value at position $j$ must change from $i$ to $i'$, which as words exchanges the values $i$ and $i'$. So rows of an ensemble move by right multiplication and columns by left multiplication. In particular, if $\T_{i,j}$ is determined and we know $\T_{i,j} \sim \T_{i,j'}$ (resp. $\T_{i,j} \sim \T_{i',j}$), then $\T_{i,j'} = \T_{i,j}(j\ j')$ (resp. $\T_{i',j} = (i\ i')\T_{i,j}$) is determined. We write $\T_{i,j} \sim \T_{i,j'} = w$ and call $w$ the \newword{candidate} for that cell linked to $\T_{i,j}$. For example, from $\T_{1,1} = \uu{1}234$ and $\T_{1,1} \sim \T_{1,3}$ we get $\T_{1,3} = \mathrm{id}\cdot(1\ 3) = 32\uu{1}4$.

\begin{example}[A running example]
\label{ex:running}
The $S_4$-ensemble of \cref{tab:running} will follow us through this section: every gadget below is illustrated on it. The identity appears three times, at the diagonal cells $(1,1)$, $(2,2)$, $(3,3)$. The entry $z := 231\uu{4}$ at $(4,4)$ appears twice, and $p_1 := z\,(1\ 4) = 4312$ appears four times, at the cells $(1,3)$, $(2,4)$, $(3,2)$, $(4,1)$. The row link of $e = \T_{1,1}$ is at column $4$, where $\T_{1,4} = e\,(1\ 4) = 423\uu{1}$, and its column link is at row $3$, where $\T_{3,1} = (1\ 3)\,e = \uu{3}214$.
\end{example}

\begin{table}[h]
\centering
$$\begin{matrix}
\uu{1}234 & 4\uu{1}32 & 43\uu{1}2 & 423\uu{1}\\
\uu{2}314 & 1\uu{2}34 & 13\uu{2}4 & 431\uu{2}\\
\uu{3}214 & 4\uu{3}12 & 12\uu{3}4 & 421\uu{3}\\
\uu{4}312 & 1\uu{4}32 & 13\uu{4}2 & 231\uu{4}\\
\end{matrix}$$
\caption{The running example of \cref{sec:algebra}: a valid $S_4$-ensemble in which $e$ appears three times (the diagonal cells of $[3]$), $z = 231\uu{4}$ twice, and $p_1 = z\,(1\ 4) = 4312$ four times. Its blocking graph is triangle-free with unique sink $p_1$.}
\label{tab:running}
\end{table}

Next we show that compatibility between two cells can be checked by studying the quotient of two permutations.

\begin{lemma}
\label{lem:compat}
Let $u, v$ be entries with underlines in columns $j, j'$ respectively. Then $u$ and $v$ are compatible if and only if every nontrivial cycle of $v^{-1}u$ meets $\{j, j'\}$.
\end{lemma}
\begin{proof}
An incompatibility between $u$ and $v$ is a set $I$ with $j, j' \notin I$, $u(I) = v(I)$, and $u|_I \neq v|_I$. Writing $\sigma = v^{-1}u$: the condition $u(I) = v(I)$ is equivalent to $\sigma(I) = I$, and $u|_I = v|_I$ is equivalent to $\sigma|_I = \mathrm{id}$. So an incompatibility exists if and only if $\sigma$ has an invariant set avoiding $\{j,j'\}$ on which it acts nontrivially. A minimal such set is a nontrivial cycle of $\sigma$ avoiding $\{j,j'\}$.
\end{proof}

For example, $\uu{1}432$ and $1\uu{4}23$ have quotient the transposition $(3\ 4)$, which avoids the underline positions $\{1,2\}$. They are incompatible. The criterion handles all $k$-compatibility conditions at once, and yields two consequences that we use constantly.

\begin{corollary}[Line Lemma]
\label{cor:linelemma}
Let $u, v$ be two distinct entries of an $S_n$-ensemble.
\begin{enumerate}
\item If $u$ and $v$ share a column $j$ (resp.\ a row), they are compatible if and only if their quotient $v^{-1}u$ is a single cycle, which then passes through $j$ (resp.\ through both underline positions), in particular they are linked if and only if the quotient is a transposition.
\item If $v^{-1}u$ is \emph{not} a single cycle, then $u$ and $v$ share neither a row nor a column.
\end{enumerate}
\end{corollary}
\begin{proof}
For entries $u, v$ in column $j$ with values $i \neq i'$ there, $\sigma := v^{-1}u$ satisfies $\sigma(j) = v^{-1}(i) \neq j$, so $j$ lies on a nontrivial cycle of $\sigma$, by \cref{lem:compat} no other nontrivial cycle may exist. For a row, $\sigma(j) = j'$, so $j$ and $j'$ lie on the same cycle and again no other may exist. For the linkage statement, a link is a one-sided multiplication by a transposition, so the quotient of a linked pair is a transposition, and conversely a transposition quotient along a row or column is realized by exactly such a multiplication. 

For the second claim, two entries in a common row or column are, being entries of an ensemble, compatible. So by the first claim, their quotient is a single cycle (or the identity). A non-single-cycle quotient therefore forbids a shared row or column.
\end{proof}

\begin{example}
In \cref{tab:running}: the entries $\T_{1,1} = \uu{1}234$ and $\T_{1,3} = 43\uu{1}2$ share row $1$, and their quotient $(1\ 3\ 2\ 4)$ is a single cycle passing through both underline columns. In the other direction, $4\uu{1}32$ at $(1,2)$ and $z = 2314$ (at $(2,1)$ and $(4,4)$) have quotient $(1\ 4)(2\ 3)$, two disjoint transpositions, and accordingly these entries share no row and no column.
\end{example}

\subsection{Symmetries}

\begin{lemma}
\label{lem:symmetry}
The following operations map $S_n$-ensembles to $S_n$-ensembles and preserve multiplicities of entries:
\begin{enumerate}
\item (two-sided action) for $\pi, \rho \in S_n$, the table $\T'_{\pi(i), \rho(j)} = \pi\, \T_{i,j}\, \rho^{-1}$,
\item (transpose-inverse) the table $\T'_{j,i} = (\T_{i,j})^{-1}$.
\end{enumerate}
Moreover, the two-sided action is transitive on pairs. In particular one may always normalize a chosen entry to be the identity permutation at a diagonal cell.
\end{lemma}
\begin{proof}
Let us start with the first claim. Forced entry condition is preserved, as $\T'_{\pi(i),\rho(j)}(\rho(j)) = \pi(\T_{i,j}(j)) = \pi(i)$. Linked pairs map to linked pairs since $(\pi u \rho^{-1})^{-1}(\pi v \rho^{-1}) = \rho\, u^{-1}v\, \rho^{-1}$ is a transposition iff $u^{-1}v$ is. Compatibility is preserved since quotients are conjugated by $\rho$ while underline positions are relabeled by $\rho$.

Now we look at the second claim. Forced entry condition is preserved, as if $u = \T_{i,j}$ then $u^{-1}(i) = j$, so $u^{-1}$ is admissible at cell $(j,i)$. Row links become column links and vice versa since $u \sim v$ iff $u^{-1} \sim v^{-1}$ (the quotients $u^{-1}v$ and $u v^{-1}$ are conjugate transpositions). For compatibility, suppose $I$ witnesses an incompatibility of $u, v$ (columns $j, j'$), and let $J = u(I) = v(I)$. Then $J$ avoids the underline positions $i, i'$ of $u^{-1}, v^{-1}$ (as $u^{-1}(i) = j \notin I$ and $v^{-1}(i') = j' \notin I$), $u^{-1}(J) = I = v^{-1}(J)$, and $u^{-1}|_J = v^{-1}|_J$ would force $u|_I = v|_I$, so $J$ witnesses an incompatibility of the inverses.

For transitivity, given a target entry $g$ at cell $(i,j)$ and a source entry $h$ at $(i_0, j_0)$: choose $\pi$ with $\pi(i_0) = i$ and set $\rho = g^{-1}\pi h$, then $\rho(j_0) = g^{-1}\pi(i_0) = g^{-1}(i) = j$ and $\pi h \rho^{-1} = g$.
\end{proof}

\begin{example}
Again take a look at \cref{tab:running}. Notice we put the identity at $(1,1)$, which is what we tend to do as an application of \cref{lem:symmetry}. Under transpose-inverse the table becomes another valid ensemble, whose multiplicity four entry is $p_1^{-1} = 3421$.
\end{example}

From \cref{lem:symmetry}, we can extract the following tool.

\begin{corollary}
\label{cor:transport}
Let $P$ be a set of prescribed entries and let $\sigma$ be a symmetry of \cref{lem:symmetry} with $\sigma(P) = P$. If every $S_n$-ensemble containing $P$ has the entry $v$ at the cell $c$, then every such ensemble has the entry $\sigma(v)$ at the cell $\sigma(c)$.
\end{corollary}
\begin{proof}
For an ensemble $\T$ containing $P$, the table $\sigma^{-1}(\T)$ is an ensemble containing $\sigma^{-1}(P) = P$, so its entry at $c$ is $v$. That entry is $\sigma^{-1}$ applied to the entry of $\T$ at $\sigma(c)$, whence $\T_{\sigma(c)} = \sigma(v)$.
\end{proof}

\subsection{The link graph of a line}

Two entries sharing a line need not be linked to each other, but the Line Lemma (\cref{cor:linelemma}) constrains where each of them may find its link.

\begin{lemma}
\label{lem:linkrestrict}
Let $p_a = \T_{i,a}$ and $p_b = \T_{i,b}$ be two entries of a common row, compatible but not linked, and set $q := p_b^{-1}p_a$, by \cref{cor:linelemma} $q$ is a single cycle through $a$ and $b$, and in fact $q(a) = b$ (since $p_a(a) = i = p_b(b)$ gives $q(a) = p_b^{-1}(i) = b$). Then the row link of $p_a$ lies either at column $q^{-1}(a)$ or at a column outside $\mathrm{supp}(q)$. The row link of $p_b$ lies either at column $q(b)$ or outside $\mathrm{supp}(q)$. The dual statement holds for two entries of a common column, with $q$ replaced by the left quotient $p_b p_a^{-1}$.
\end{lemma}
\begin{proof}
Let $q$ be a $k$-cycle. A row link of $p_a$ at a support column $x = q^m(a)$ with $1 \le m \le k-1$ places the candidate $p_a\,(a\ x)$. For $m = 1$ we have $x = q(a) = b$, the column of $p_b$ itself, so this link would make $p_a$ and $p_b$ linked, contrary to hypothesis. For $2 \le m \le k-2$, the quotient of the candidate with $p_b$ is $q\,(a\ x)$, which splits into two disjoint nontrivial cycles, of lengths $m$ and $k-m$. As this candidate would share row $i$ with $p_b$, the Line Lemma (\cref{cor:linelemma}) forbids it. Only $m = k-1$, i.e.\ $x = q^{-1}(a)$, survives among support columns. The remaining possibility is an off-support column. The statement for $p_b$ is the same computation with the roles reversed, and the column statement follows by symmetry in \cref{lem:symmetry}.
\end{proof}

\begin{example}
In row $2$ of \cref{tab:running}, take $p_a = \T_{2,2} = 1\uu{2}34$ and $p_b = \T_{2,1} = \uu{2}314$, so $q = p_b^{-1}p_a = (1\ 3\ 2)$ with $q(2) = 1$, as the lemma prescribes. The row link of $p_a$ must lie at $q^{-1}(2) = 3$ or at the off-support column $4$, in the table it is at column $3$.
\end{example}

At $n = 4$ there are few off-support columns to spare, so we get a stronger claim.

\begin{corollary}[Forced Bridge]
\label{cor:bridge}
Let $p_a = \T_{i,a}$, $p_b = \T_{i,b}$ be entries of a common row of an $S_4$-ensemble whose quotient $q = p_b^{-1}p_a$ is the $3$-cycle on $\{a,b,c\}$. Then the \newword{bridge}
$$\T_{i,c} = p_a\,(a\ c) = p_b\,(b\ c)$$
is an entry of the row (its value at $c$ is $i$). The dual statement holds in a column.
\end{corollary}
\begin{proof}
Here $\mathrm{supp}(q) = \{a,b,c\}$ with $q^{-1}(a) = c$, and at $n = 4$ there is a single off-support column $d$. By \cref{lem:linkrestrict}, the row link of $p_a$ is at $c$ or $d$, and likewise for $p_b$. If both linked at $d$ then $\T_{i,d} = p_a(a\ d) = p_b(b\ d)$, forcing $q = (b\ d)(a\ d)$, contradicting $\mathrm{supp}(q) = \{a,b,c\}$. So at least one of $p_a, p_b$ links at $c$, its partner there is the displayed bridge, and $p_a(a\ c) = p_b(b\ c)$ since $(b\ c)(a\ c) = (a\ c\ b) = q$. Its value at column $c$ is $p_a(a) = i$.
\end{proof}

\begin{example}
Continuing in row $2$ of \cref{tab:running}: the quotient $q = (1\ 3\ 2)$ of the previous example has support $\{1,2,3\}$, and the bridge sits at column $3$:
$\T_{2,3} = p_a\,(2\ 3) = p_b\,(1\ 3) = 13\uu{2}4$.
\end{example}

\begin{corollary}[Completion]
\label{cor:completion}
Let $p_a = \T_{i,a}$, $p_b = \T_{i,b}$ be entries of a common row of an $S_4$-ensemble whose quotient $q = p_b^{-1}p_a$ is a $4$-cycle. Write $c = q(b)$ and $d = q(c)$, so $q = (a\ b\ c\ d)$ and $\{a,b,c,d\} = [4]$. Then the row is completely determined:
$$\T_{i,c} = p_b\,(b\ c), \qquad \T_{i,d} = p_a\,(a\ d).$$
The dual statement holds in a column.
\end{corollary}
\begin{proof}
At $n = 4$ a $4$-cycle has no off-support column, so by \cref{lem:linkrestrict} the row link of $p_a$ is at $q^{-1}(a) = d$ and that of $p_b$ is at $q(b) = c$. These are $\T_{i,d} = p_a(a\ d)$ and $\T_{i,c} = p_b(b\ c)$, completely determining the row.
\end{proof}

\begin{example}
In row $1$ of \cref{tab:running} we have $p_a = \T_{1,1} = \uu{1}234$ and $p_b = \T_{1,3} = 43\uu{1}2$, with quotient the $4$-cycle $q = p_b^{-1}p_a = (1\ 3\ 2\ 4)$, so $c = q(3) = 2$ and $d = q(2) = 4$. \cref{cor:completion} forces the rest of the row: $\T_{1,2} = p_b\,(3\ 2) = 4\uu{1}32$ and $\T_{1,4} = p_a\,(1\ 4) = 423\uu{1}$, as the table shows.
\end{example}

Thus, at $n = 4$, a pair of entries sharing a line does precisely one of four things, according to the cycle type of its quotient: a transposition means the two are linked (\cref{cor:linelemma}(1)). A $3$-cycle forces the bridge (\cref{cor:bridge}). A $4$-cycle forces both remaining cells of the line (\cref{cor:completion}), and a product of two disjoint transpositions is impossible in a shared line (\cref{cor:linelemma}(2)). Every quotient type is maximally rigid.

The bridge and completion rules combine into a structural statement about the \newword{link graph} of a line: the graph on the four positions of a row (resp.\ the four rows of a column) in which two positions are adjacent when the corresponding entries are linked.

\begin{corollary}[Line Dichotomy]
\label{cor:linetree}
The link graph of every line of an $S_4$-ensemble is a spanning tree: a star or a path.
\end{corollary}
\begin{proof}
 Link graphs of lines never contain cycles, for any $n$: going around a cycle of positions $c_1, c_2, \dots, c_k, c_1$ multiplies the entry on the right by $(c_1c_2)(c_2c_3)\cdots(c_kc_1)$, which sends $c_k$ to $c_2$ and hence is not the identity. By the linkage axiom no vertex is isolated, so at $n=4$ the link graph is either a spanning tree ($3$ edges) or a perfect matching ($2$ edges), and it remains to exclude the matching, say with edges $a \sim b$ and $c \sim d$.

Assume the line is a row. Consider the pair $p_a = \T_{i,a}$, $p_c = \T_{i,c}$ not linked, so by the Line Lemma (\cref{cor:linelemma}) the quotient $q = p_c^{-1}p_a$ is a $3$-cycle or a $4$-cycle through $a$ and $c$ with $q(a) = c$. If $q$ is a $3$-cycle with support $\{a,c,x\}$, then by Forced Bridge (\cref{cor:bridge}) the entry at column $x$ equals both $p_a(a\ x)$ and $p_c(c\ x)$. Whichever of $x = b, d$ holds, this makes the entry at $x$ linked to \emph{both} $p_a$ and $p_c$, giving a vertex of degree two: a third edge. If $q$ is a $4$-cycle, then by Completion (\cref{cor:completion}) $p_a$ links at $q^{-1}(a)$ and $p_c$ at $q(c)$. The matching forces $q^{-1}(a) = b$ and $q(c) = d$, so $q = (a\ c\ d\ b)$, and the quotient of the two remaining entries is
$$p_d^{-1}p_b = (c\ d)\,q\,(a\ b) = (b\ d),$$
a transposition: $p_b$ and $p_d$ are linked, again a third edge. Either way the link graph has more than two edges, contradicting the matching. So the link graph is a spanning tree, and the two isomorphism types of trees on four vertices are the star and the path.

Column case follows by symmetry in \cref{lem:symmetry}.
\end{proof}

\begin{example}
All eight lines of \cref{tab:running} have path link graphs. Row $1$'s path visits the columns in the order $1\!-\!4\!-\!2\!-\!3$. Stars do occur: in the rigid ensemble $E^*_4$ of \cref{tab:estar4}, row $4$ and column $4$ are stars centered at position $4$.
\end{example}

Studying the distance within a link graph gives us the following tool.

\begin{lemma}
\label{lem:pathmetric}
Let $u, v$ be two entries of a line of an $S_4$-ensemble, at distance $d$ in its link graph (a star or a path, by \cref{cor:linetree}). Then the quotient from $u$ to $v$, namely $u^{-1}v$ for a row and $u\,v^{-1}$ for a column, is the single cycle spelled by the successive positions along the tree path from $u$ to $v$: a cycle of length $d + 1$ supported on those positions. 
\end{lemma}
\begin{proof}
Entries adjacent in the link graph differ by the transposition of their two positions, so the quotient from $u$ to $v$ along the tree path $d_0 - d_1 - \cdots - d_d$ is $(d_0\,d_1)(d_1\,d_2)\cdots(d_{d-1}\,d_d) = (d_0\ d_1\ \cdots\ d_d)$.
\end{proof}

Hence the distance between two entries of a line is the length of their quotient cycle minus one. A star has diameter $2$, so a $4$-cycle quotient certifies that the line's link graph is a path and that the two entries are its ends. 

\begin{example}
Along row $1$ of \cref{tab:running} (the path $1\!-\!4\!-\!2\!-\!3$): the entries at columns $1$ and $4$ are adjacent, with quotient the transposition $(1\ 4)$. Columns $1$ and $2$ are at distance two, with a $3$-cycle quotient, and the ends, columns $1$ and $3$, are at distance three, with the $4$-cycle quotient $(1\ 4\ 2\ 3)$ spelling the path from the entry at column $1$. The last quotient certifies, without inspecting any links, that this line is a path with these two cells as its ends.
\end{example}

\begin{remark}
\label{rem:linetreegeom}
For ensembles arising from boundary matching ensembles (\cref{thm:boundary}), the Line Dichotomy (\cref{cor:linetree}) is the combinatorial shadow of a geometric fact. A row is the family of maximal matchings of the fine mixed subdivision of a facet, and the linkage axiom for matching fields admits a strong form due to Bernstein and Zelevinsky \cite{BernsteinZelevinsky93}: the union of the $n$ maximal matchings is a spanning tree $T$, the \newword{linkage covector}, with all value-side degrees equal to $2$, from which each matching is recovered as a push $T^{j\to}$ \cite{LohoSmith20, Yao25}. Two pushes differ in a single edge exactly when the two position vertices share a neighbour in $T$, so the link graph of the row is the contraction of the covector tree onto its position vertices, connected because a spanning tree is connected. One level up, the same phenomenon governs adjacency of cells: the facet-adjacency graph of the cells around a lattice point of a fine mixed subdivision is a spanning tree \cite[Lemma 4.10]{OYoo10} (see also \cite[\S 5]{Yao25}), the cells being encoded by spanning trees \cite{ArdilaBilley07} and the lattice points by matchings and topes \cite{LohoSmith20, GalashinNenashevPostnikov23}. Since an $S_n$-ensemble carries only the top layer of this structure (\cref{rem:loosedef}), we do not rely on the geometry: the proof above uses only the ensemble axioms.
\end{remark}


\subsection{The blocking graph}
\label{subsec:blocking}

We now attach to every ensemble a directed graph that turns out to be a useful tool.

\begin{definition}
\label{def:blocking}
The \newword{blocking graph} of an $S_n$-ensemble $\T$ is the directed graph with vertex set $S_n$ and a labeled edge $u \to_j v$ whenever $\T_{u(j),\,j} = v$ and $v \neq u$, that is, whenever the table cell of $u$ in column $j$ is occupied by a different permutation. We say that $v$ \newword{blocks} $u$ at column $j$.
\end{definition}

Note that the vertex set is all of $S_n$, not merely the entries of $\T$. The entries are recovered as the targets of edges. We write $\mathrm{mult}(u)$ for the number of times $u$ appears in $\T$.

\begin{example}[The blocking graph of an $S_3$-ensemble]
\label{ex:blockgraph}
\cref{fig:blockgraph} draws the blocking graph of the $S_3$-ensemble of \cref{tab:s3ensemble}, restricted to the five permutations that occur as entries (the sources of $S_3$, being non-entries, are omitted). Each edge $u \to_j v$ records that the table cell of $u$ in column $j$, the cell $(u(j), j)$, is occupied by the entry $v$, for instance $132 \to_2 231$ because the table cell of $132$ in column $2$ is $(3,2)$, where \cref{tab:s3ensemble} has $231$. The permutation $231$, which appears three times, is the unique sink (out-degree $3 - \mathrm{mult}(231) = 0$), and every other vertex has a directed path into it. The graph is acyclic, as \cref{thm:s3ensemble} requires of an ensemble with a permutation of full multiplicity.
\end{example}

\begin{figure}[h]
\centering
\input{tikz/blockgraph_s3}
\caption{The blocking graph of the $S_3$-ensemble of \cref{tab:s3ensemble}, on its five entry-permutations. Edge labels are blocking columns. The boxed vertex $231$ is the unique sink, of multiplicity three.}
\label{fig:blockgraph}
\end{figure}

Next we study how the degree of each vertex behaves in the blocking graph.

\begin{proposition}
\label{prop:blockdeg}
In the blocking graph of an $S_n$-ensemble $\T$:
\begin{enumerate}
\item every vertex $u$ has out-degree $n - \mathrm{mult}(u)$, in particular the sinks are exactly the permutations of multiplicity $n$,
\item every vertex $v$ has in-degree $\mathrm{mult}(v)\cdot\big((n-1)! - 1\big)$, in particular the sources are exactly the permutations that do not appear in $\T$, and every vertex lying on a directed cycle is an entry of $\T$,
\item the total number of edges is $n \cdot n! - n^2$.
\end{enumerate}
\end{proposition}
\begin{proof}
For the first claim, each of the $n$ table cells $(u(j), j)$ of $u$ either carries $u$ itself or contributes the edge $u \to_j \T_{u(j),j}$, and $\mathrm{mult}(u)$ counts the former. 

For the second claim, fix an occurrence of $v$ at the cell $(i,j)$. The edges into $v$ through this cell are $u \to_j v$ for the permutations $u \neq v$ with $u(j) = i$, of which there are $(n-1)! - 1$. Distinct occurrences of $v$ lie in distinct columns, so these edge sets are disjoint. A vertex on a directed cycle has an incoming edge, hence positive multiplicity. 

For the last claim, sum (1) over $u \in S_n$, using $\sum_u \mathrm{mult}(u) = n^2$.
\end{proof}

The compatibility axiom constrains which cycles can occur: the shortest possible ones are triangles.

\begin{lemma}
\label{lem:girth}
The blocking graph of an $S_n$-ensemble contains no directed cycle of length two.
\end{lemma}
\begin{proof}
Suppose $u \to_j v \to_{j'} u$ with $u \neq v$. The edge $u \to_j v$ says that $v$ is the entry at the cell $(u(j), j)$, so by forced entry $v(j) = u(j)$, likewise $u$ is the entry at $(v(j'), j')$ and $u(j') = v(j')$. Hence $\sigma = v^{-1}u$ fixes both $j$ and $j'$. If $j = j'$, then $u$ and $v$ are entries in the same column with the same value there, hence occupy the same cell and are equal: a contradiction. If $j \neq j'$, then $u$ and $v$ are entries with underlines in columns $j'$ and $j$, so by \cref{lem:compat} every nontrivial cycle of $\sigma$ meets $\{j, j'\}$, as $\sigma$ fixes both, $\sigma = e$ and $u = v$: again a contradiction.
\end{proof}

\begin{example}[The blocking graph of the running example]
\label{ex:runblock}
In \cref{tab:running}, out-degrees are $4 - \mathrm{mult}$: the seven entries of multiplicity one have out-degree $3$. The twice-appearing $z$ has out-degree $2$. The identity has out-degree $1$, its single edge being $e \to_4 z$, and $p_1$, of multiplicity four, is a sink. The in-degree of $p_1$ is $4\cdot(3! - 1) = 20$, and the graph has $4 \cdot 4! - 4^2 = 80$ edges in total, as \cref{prop:blockdeg} predicts. The graph is triangle-free, and it has a sink, as \cref{thm:s4forward} promises of every triangle-free $S_4$-ensemble.
\end{example}

\subsection{Blocker tools}
\label{subsec:blocker}

The final tools record what the linkage axiom forces in the neighborhood of a permutation of near-full multiplicity. These \newword{blocker} lemmas, all valid for every $n$, are the opening moves of the promotion argument of \cref{sec:main}, and we expect them to be the correct way to tackle general $n$ cases. Throughout, $C_w \subseteq [n]$ denotes the set of columns in which $w$ occupies its table cells, so $|C_w| = \mathrm{mult}(w)$. The entry $\T_{w(c_0),\,c_0}$ at the table cell of $w$ in a missing column $c_0$ is its \newword{blocker} there.

\begin{lemma}[Blocker Lemma]
\label{lem:blocker}
Let $w$ have multiplicity $n-1$ in an $S_n$-ensemble, with missing column $c_0$ and blocker $z := \T_{w(c_0),\,c_0} \neq w$. Then $z^{-1}w$ is a single $(n-1)$-cycle supported exactly on $C_w$.
\end{lemma}
\begin{proof}
Write $\sigma = z^{-1}w$. Forced entry gives $z(c_0) = w(c_0)$, so $\sigma$ fixes $c_0$. For each $c \in C_w$, the entries $z$ (column $c_0$) and $w$ at $(w(c),c)$ (column $c$) are compatible, so by \cref{lem:compat} every nontrivial cycle of $\sigma$ meets $\{c_0, c\}$, as this holds for every $c \in C_w$ and $\sigma$ fixes $c_0$, every nontrivial cycle of $\sigma$ contains all of $C_w$. Two disjoint cycles cannot both do so, so $\sigma$ is a single cycle with support exactly $C_w$.
\end{proof}

\begin{example}
In \cref{tab:running}, $w = e$ has multiplicity $3$ with $C_e = \{1,2,3\}$ and missing column $c_0 = 4$, its blocker is $z = \T_{4,4} = 231\uu{4}$, and $z^{-1}e = z^{-1} = (1\ 3\ 2)$ is a single $3$-cycle supported exactly on $C_e$.
\end{example}

\begin{lemma}[Triangle Filter]
\label{lem:trifilter}
Let $w$ have multiplicity $n-1$ in a triangle-free $S_n$-ensemble, with missing column $c_0$ and blocker $z := \T_{w(c_0),\,c_0} \neq w$. Then no blocking target $g$ of $z$ satisfies $g(c) = w(c)$ for an occupied column $c \in C_w$.
\end{lemma}
\begin{proof}
If $g(c) = w(c)$, then the copy of $w$ at $(w(c),c)$ sits at $g$'s table cell in column $c$, giving the edge $g \to_c w$, together with $w \to_{c_0} z$ and $z \to g$ this closes a directed triangle.
\end{proof}

\begin{example}
Continuing in \cref{tab:running}, whose blocking graph is triangle-free (\cref{ex:runblock}): the blocking targets of $z$ are given by the two edges $z \to_2 p_1$ and $z \to_3 p_1$, and $p_1 = 4312$ has no fixed point in $[3]$: it agrees with the multiplicity-$3$ permutation $e$ at no occupied column, exactly as the filter demands.
\end{example}

\section{The characterization for $n = 4$}
\label{sec:main}

The arguments of this section are specific to $n = 4$, though every tool we use (the compatibility criterion of \cref{lem:compat}, the Line Dichotomy of \cref{cor:linetree}, and the blocker tools of \cref{subsec:blocker}) rests on general-$n$ statements. 

\subsection{Directed triangles of the blocking graph}
\label{subsec:triangles}

By \cref{lem:girth} the shortest possible directed cycle in a blocking graph is a triangle. For $n = 4$ we show that a directed triangle has a completely rigid local shape, and it is this shape that obstructs full multiplicity. We take the triangle as the primary object and call it a \newword{pattern}. \cref{cor:4d3mixed} shows how the pattern takes over the obstructing role that the cyclic pattern $ababab$ plays in two dimensions.

\begin{definition}
\label{def:pattern}
A \newword{pattern} in an $S_4$-ensemble $\T$ is a directed triangle of its blocking graph: three entries $u_1 \to_{j_1} u_2 \to_{j_2} u_3 \to_{j_3} u_1$. We call $\T$ \newword{triangle-free} if its blocking graph contains no directed triangle.
\end{definition}

The next lemma shows the triangle is forced into a single form, from which its whole structure follows.

\begin{lemma}[Normal form of a pattern]
\label{lem:patternform}
Let $u_1 \to_{j_1} u_2 \to_{j_2} u_3 \to_{j_3} u_1$ be a pattern in an $S_4$-ensemble, and let $j_4$ be the missing column. Then:
\begin{enumerate}
\item the columns $j_1, j_2, j_3$ are distinct, and each $u_a$ is an entry with underline in column $j_{a-1}$,
\item the quotients $\tau_a := u_{a+1}^{-1}u_a$ are the three $3$-cycles through $j_4$: with $(j_1,j_2,j_3)=(1,2,3)$ normalized, $\tau_1 = (2\,3\,4)$, $\tau_2 = (1\,4\,3)$, $\tau_3 = (1\,2\,4)$,
\item writing $w := u_1\,(j_2\ j_4)$ and $v := (j_3, j_1, j_2, j_4)$, the three entries are the table cells $(w(v_c), v_c)$ of $w$, with the entry in column $v_c$ equal to $w\,(v_{c-1}\ v_4)$ (indices cyclic in $\{1,2,3\}$). Conversely, for any $w, v \in S_4$ these three cells form a pattern.
\end{enumerate}
In particular the three entries are pairwise compatible, and there are exactly $4!\cdot 4!/3 = 192$ patterns, forming a single orbit under the symmetries of \cref{lem:symmetry}.
\end{lemma}
\begin{proof}
Read the indices cyclically, so the edge $u_a \to_{j_a} u_{a+1}$ says $u_{a+1}$ is the entry at $(u_a(j_a), j_a)$, each $u_a$ is then an entry with underline in column $j_{a-1}$, and by forced entry
\begin{equation}
u_{a+1}(j_a) = u_a(j_a). \tag{$*$}
\end{equation}

We first check that the labels are distinct. If $j_a = j_{a+1}$, then $u_{a+1}, u_{a+2}$ are entries in column $j_a$ taking the same value there by $(*)$, so they share a cell and $u_{a+1} = u_{a+2}$, a contradiction. So $j_1, j_2, j_3$ are distinct, by \cref{lem:symmetry} (a column relabeling) assume $(j_1,j_2,j_3) = (1,2,3)$, $j_4 = 4$.

We next compute the quotients. Set $\tau_a = u_{a+1}^{-1}u_a$, so $\tau_3\tau_2\tau_1 = e$. By $(*)$, $\tau_a$ fixes $j_a$, and $u_a, u_{a+1}$ have underlines in columns $j_{a-1}, j_a$, so by \cref{lem:compat} every nontrivial cycle of $\tau_a$ meets $\{j_{a-1}, j_a\}$, whence (as $\tau_a$ fixes $j_a$ and $\tau_a \neq e$) $\tau_a$ moves $j_{a-1}$. Evaluating $\tau_3\tau_2\tau_1 = e$ at $3$: since $\tau_3$ fixes $3$ we need $\tau_2(\tau_1(3)) = 3$, and $\tau_1(3) \in \{2,4\}$ with $\tau_1(3) = 2$ forcing the contradiction $\tau_2(2) = 3$, so $\tau_1(3) = 4$, $\tau_2(4) = 3$. The two cyclic shifts of the identity give likewise $\tau_2(1) = 4, \tau_3(4) = 1$ and $\tau_3(2) = 4, \tau_1(4) = 2$. These determine $\tau_1 = (2\,3\,4)$, $\tau_2 = (1\,4\,3)$, $\tau_3 = (1\,2\,4)$.

Finally we describe what the pattern looks like within a table. Set $w = u_1(2\,4)$, $v = (3,1,2,4)$. Then $u_1 = w(v_3\,v_4)$ sits in column $v_1 = 3$, and using $u_{a+1} = u_a\tau_a^{-1}$, $u_2 = w(3\,4) = w(v_1\,v_4)$ and $u_3 = w(1\,4) = w(v_2\,v_4)$ in columns $v_2, v_3$, each $u_a$ occupies the table cell $(w(v_a),v_a)$ of $w$ since $v_{a+1} \notin \{v_{a-1}, v_4\}$. Undoing the normalization gives the general formula. Conversely, given $w, v$, the entry $w(v_{c-1}\,v_4)$ at $(w(v_c),v_c)$ has value $w(v_{c+1})$ at column $v_{c+1}$, which is where $w(v_c\,v_4)$ sits, so consecutive cells form edges and the three cells are a directed triangle. 

For compatibility and the count: the quotient of the entries in columns $v_a, v_b$ is $(v_{a-1}\,v_4)(v_{b-1}\,v_4)$, a $3$-cycle whose support meets $\{v_a,v_b\}$ in exactly one point, so \cref{lem:compat} applies, and the patterns indexed by $(w,v)$ and $(w, v\circ\gamma)$ for cyclic shifts $\gamma$ of $(v_1,v_2,v_3)$ coincide, giving $4!\cdot4!/3 = 192$, a single orbit by transitivity of the two-sided action.
\end{proof}

\begin{example}
\label{ex:choice}
Consider the system of permutations $1234,1324,1423,3412,4132,4213$ on the edges $1'2',1'3',1'4',2'3',2'4',3'4'$. Multiple entries admit two choices (\cref{tab:acycexample}):
\begin{table}[h]
\centering
$$\begin{matrix}
\uu{1}234 & 2\uu{1}34 & 32\uu{1}4 & 423\uu{1}\\
\uu{2}134 & 4\uu{2}31,\textcolor{red}{3\uu{2}14} & 31\uu{2}4 & 341\uu{2},\textcolor{red}{413\uu{2}}\\
\uu{3}214 & 4\uu{3}21,\textcolor{red}{2\uu{3}14} & \textcolor{red}{42\uu{3}1},21\uu{3}4 & 421\uu{3}\\
\uu{4}231 & 2\uu{4}31 & 32\uu{4}1 & 321\uu{4},\textcolor{red}{213\uu{4}}\\
\end{matrix}$$
\caption{The table from $1234,1324,1423,3412,4132,4213$ has a choice containing a pattern.}
\label{tab:acycexample}
\end{table}
Picking the red entries yields an $S_4$-ensemble that is not triangle-free, since $3\uu{2}14, 42\uu{3}1, 213\uu{4}$ forms a pattern. Each choice of entries corresponds to a choice of fine mixed subdivision of the corresponding facet of $4\Delta_3$. The rows record the facet subdivisions as in \cref{thm:boundary}.
\end{example}

\subsection{The multiplicity theorem}
\label{subsec:mult3}

We do not need triangle-free for the following claim.

\begin{theorem}[Multiplicity theorem]
\label{thm:mult3}
Every $S_4$-ensemble contains a permutation that appears at least $3$ times.
\end{theorem}

\begin{proof}
Suppose every permutation has multiplicity capped at $2$. By the Line Dichotomy (\cref{cor:linetree}) every line's link graph is a star or a path.

We first show that no line is a star. Suppose some row's link graph is a star (a column star reduces to this by transpose-inverse). By \cref{lem:symmetry} normalize its center to $\T_{1,1} = \uu1234$, the three leaves are then $\T_{1,c} = e\,(1\ c)$, so row $1$ is
$$\uu1234, \quad 2\uu134, \quad 32\uu14, \quad 423\uu1.$$
The residual symmetry is the diagonal $\mathrm{Stab}(1) \cong S_3$, acting simultaneously on rows, columns and values $\{2,3,4\}$.

By linkage, $e$ at $(1,1)$ has a column link, necessarily the value $(1\ r)$ at $(r,1)$, a value already present at $(1,r)$. By the residual symmetry take $r = 2$: the entry $\uu2134$ at $(2,1)$ caps $2134$.

We next compute the diagonal entries. For $c \in \{2,3,4\}$, compatibility with row $1$ alone forces
$$\T_{c,c} \in \{e\} \cup \{(1\ c') : c' \neq c\}.$$
For $c = 2$ the six candidates and their killers are: $1\uu234$ and $3\uu214, 4\uu231$ survive, while $1\uu243$ dies against $\uu1234$ (quotient $(3\,4)$ avoids $\{2,1\}$), $3\uu241$ against $423\uu1$ (quotient $(1\,3)$ avoids $\{2,4\}$), and $4\uu213$ against $32\uu14$ (quotient $(1\,4)$ avoids $\{2,3\}$). The cells $(3,3)$ and $(4,4)$ are the images of this computation under the residual symmetry.

Row $1$ together with the capped entry $\uu2134$ uses $e$, $3214$, $4231$ once each and $2134$ twice. By the cap, the three diagonal cells $(2,2), (3,3), (4,4)$ therefore carry the three values $e, (1\,3), (1\,4)$ bijectively.

But such a diagonal is too rigid. Say $(1\ x)$ sits at $(y,y)$ and $e$ at $(z,z)$. Their quotient is $(1\ x)$ itself, and compatibility requires it to meet $\{y,z\}$, i.e.\ $x \in \{y, z\}$. Since $(1\ x)$ does not fix $x$ we have $y \neq x$, hence $z = x$: $e$ must occupy $(x,x)$. But the bijective placement above puts both $(1\,3)$ and $(1\,4)$ on the diagonal, forcing $e$ at $(3,3)$ and at $(4,4)$ simultaneously: a contradiction. 

It remains to rule out the case where every line is a path. Normalize row $1$ as follows: put the value $e$ at one end of its path at cell $(1,1)$ (\cref{lem:symmetry}), and use the residual diagonal $S_3$ to relabel the columns so that the path visits $1\!-\!2\!-\!3\!-\!4$. The row is then
$$\uu1234, \quad 2\uu134, \quad 23\uu14, \quad 234\uu1,$$
the prefix cycles $e, (1\,2), (1\,2\,3), (1\,2\,3\,4)$. This configuration retains one involution, the \newword{reversal} $(\pi, \rho) = ((2\,4),\ (1\,4)(2\,3))$ of \cref{lem:symmetry}: it preserves row $1$, exchanges the cells $(2,1) \leftrightarrow (4,4)$, and acts on values by $e \leftrightarrow 2341$ and $2134 \leftrightarrow 2314$.

Compatibility with row $1$ alone pins the two exchanged cells inside row $1$'s own value set:
$$\T_{2,1} \in \{\uu2134,\ \uu2314,\ \uu2341\}, \qquad \T_{4,4} \in \{123\uu4,\ 213\uu4,\ 231\uu4\}.$$
(At $(2,1)$: $\uu2143$ dies against $\uu1234$, $\uu2413$ against $23\uu14$, $\uu2431$ against $234\uu1$, dually at $(4,4)$.) Each of these two cells therefore caps a row-$1$ value. The cap forbids their being equal, and the direct compatibility check between the two cells kills the pairs $(2134, 2314)$ and $(2314, 2134)$, whose quotient $(2\,3)$ avoids $\{1,4\}$. Five ordered pairs survive, and the reversal pairs them as
$$(2134,\, e) \leftrightarrow (2341,\, 2314), \qquad (2314,\, e) \leftrightarrow (2341,\, 2134), \qquad (2341,\, e) \text{ fixed},$$
so we may assume $\T_{4,4} = 123\uu4$: the value $e$ is capped at $(1,1)$ and $(4,4)$.

We next show that column $4$ mirrors row $1$. The left quotient of $123\uu4$ with $234\uu1$ is the $4$-cycle $(1\,4\,3\,2)$, so by the Path Metric (\cref{lem:pathmetric}) rows $1$ and $4$ are the two \emph{ends} of column $4$'s path, and the spelled cycle prescribes the interior order: the column path is $1\!-\!2\!-\!3\!-\!4$ and column $4$ reads
$$234\uu1, \quad 134\uu2, \quad 124\uu3, \quad 123\uu4,$$
the suffix cycles $(1\,2\,3\,4), (2\,3\,4), (3\,4), e$: the mirror of row $1$'s prefix chain. The prefix row, the suffix column and the two copies of $e$ occupy seven cells, which we call the \newword{frame} (\cref{tab:frame}): it carries $e$ twice and the values $2134, 2314, 2341, 1342, 1243$ once each.

\begin{table}[h]
\centering
$$\begin{matrix}
\uu{1}234 & 2\uu{1}34 & 23\uu{1}4 & 234\uu{1}\\
\cdot & \cdot & \cdot & 134\uu{2}\\
\cdot & \cdot & \cdot & 124\uu{3}\\
\cdot & \cdot & \cdot & 123\uu{4}\\
\end{matrix}$$
\caption{The frame of the path case: the prefix row, the suffix column, and the value $e$ capped at $(1,1)$ and $(4,4)$. The nine remaining cells are unknown.}
\label{tab:frame}
\end{table}

It remains to derive a contradiction, and the main tool is the Path Metric (\cref{lem:pathmetric}): a $4$-cycle quotient certifies that the two positions are the ends of their line's path, and the spelled cycle names the neighbour of each end. We first check the cell $(2,2)$: of its six candidates, $e$ is capped, $1\uu243$ dies against $\uu1234$ at $(1,1)$ with quotient $(3\,4)$ avoiding $\{1,2\}$, $3\uu214$ dies against $123\uu4$ at $(4,4)$ with quotient $(1\,3)$ avoiding $\{2,4\}$, and $4\uu213$ dies against $134\uu2$ at $(2,4)$ with quotient $(1\,3)(2\,4)$, so $\T_{2,2} \in \{3\uu241, 4\uu231\}$.

We now show that $\T_{2,3} = 13\uu24$ is forced. The entry $134\uu2$ needs a row link: the required value at $(2,2)$ is $1\uu243$, which died above, so $\T_{2,1} = \uu2341$ or $\T_{2,3} = 13\uu24$. The entry $23\uu14$ needs a column link: the required value at $(3,3)$ is $21\uu34$, which dies against $124\uu3$ at $(3,4)$ with quotient $(1\,2)(3\,4)$, whose cycle $(1\,2)$ avoids $\{3,4\}$. So $\T_{2,3} = 13\uu24$ or $\T_{4,3} = 23\uu41$. If $\T_{2,3} \neq 13\uu24$, then $2341$ occupies $(1,4)$, $(2,1)$, $(4,3)$, exceeding the cap. So $\T_{2,3} = 13\uu24$.

The rest is forced. The right quotient of $4\uu231$ with $13\uu24$ is $(1\,4)(2\,3)$, so the Line Lemma (\cref{cor:linelemma}) bars $4231$ from row $2$ and $\T_{2,2} = 3\uu241$. The right quotient of $3\uu241$ with $13\uu24$ is the $4$-cycle $(2\,1\,4\,3)$, so by the Path Metric (\cref{lem:pathmetric}) row $2$ is the path $2\!-\!1\!-\!4\!-\!3$, whence $\T_{2,1} = 3241\,(1\,2) = \uu2341$, capping $2341$ at $(1,4)$ and $(2,1)$. Finally the left quotient of $2\uu134$ with $3\uu241$ is the $4$-cycle $(1\,4\,3\,2)$, so column $2$ is the path $1\!-\!4\!-\!3\!-\!2$, placing $(1\,4)\,2134 = 2\uu431$ at $(4,2)$ and then $(3\,4)\,2431 = 2\uu341$ at $(3,2)$, a third copy of $2341$ (\cref{tab:pathend}). Hence all lines being a path is impossible. This finishes the proof.

\begin{table}[h]
\centering
$$\begin{matrix}
\uu{1}234 & 2\uu{1}34 & 23\uu{1}4 & 234\uu{1}\\
\uu{2}341 & 3\uu{2}41 & 13\uu{2}4 & 134\uu{2}\\
\cdot & \textcolor{red}{2\uu{3}41} & \cdot & 124\uu{3}\\
\cdot & 2\uu{4}31 & \cdot & 123\uu{4}\\
\end{matrix}$$
\caption{The end of the path case: once $\T_{2,3} = 13\uu24$ is forced, row $2$ and then column $2$ are spelled out completely, and the entry at $(3,2)$ is a third copy of $2341$ (red), against the cap.}
\label{tab:pathend}
\end{table}

\end{proof}

\subsection{Promotion}
\label{subsec:promotion}

\begin{lemma}[Promotion]
\label{lem:mult3done}
If $\T$ is a triangle-free $S_4$-ensemble that has some $w \in S_4$ that appears at least $3$ times, then there is some $w' \in S_4$ that appears $4$ times.
\end{lemma}
\begin{proof}
By \cref{lem:symmetry} normalize the multiplicity-$3$ permutation to $e$, appearing at the diagonal cells $(1,1),(2,2),(3,3)$, so $C_e = \{1,2,3\}$ and the missing column is $c_0 = 4$. Write $z = \T_{4,4}$, and assume $z \neq e$ (otherwise $e$ already has multiplicity four).

We first determine the candidate list at $z$'s remaining table cells. By the Blocker Lemma (\cref{lem:blocker}), $z^{-1}e$, hence $z$, is a single $3$-cycle on $[3]$, by the residual transpose-inverse symmetry assume $z = (1\,2\,3) = 231\uu4$. By the Triangle Filter (\cref{lem:trifilter}), no blocking target of $z$ agrees with $e$ at a column of $[3]$, i.e.\ fixes a point of $[3]$. We claim the entries at $z$'s remaining table cells satisfy
$$\T_{z(j),\,j} \in \{z\} \cup \{p_a := z\,(a\ 4) : a \in [3]\setminus\{j\}\}, \qquad j \in [3].$$
Indeed, write $g := \T_{z(j),\,j}$, an entry in column $j$ with $g(j) = z(j)$, if $g = z$ we are done, so assume $g \neq z$, whence $z \to_j g$ is an edge of the blocking graph. The entries $g$ and $e$ share column $j$, so by the Line Lemma (\cref{cor:linelemma}) the quotient $e^{-1}g = g$ is a single cycle through $j$, and it fixes no point of $[3]$ by the filter above, so its support is $[3]$ or all of $[4]$. The former is impossible: the only $3$-cycles on $[3]$ are $z$ and $z^{-1}$, and $z^{-1}(j) \neq z(j) = g(j)$. So $g$ is a $4$-cycle. Finally, compatibility of $g$ (column $j$) with the blocker $z$ (column $c_0 = 4$) forces every nontrivial cycle of $z^{-1}g$ to meet $\{j,4\}$, since $z^{-1}g$ fixes $j$ it is a single cycle through $4$ avoiding $j$, with support in $([3]\setminus\{j\}) \cup \{4\}$ (size three), and $z^{-1}g$ is \emph{odd} (the product of the even $z^{-1}$, a $3$-cycle, and the odd $g$, a $4$-cycle, has sign $(+1)(-1) = -1$), so this single cycle has even length, hence length two: a transposition $(a\ 4)$ with $a \in [3]\setminus\{j\}$. Thus $g = z\,(a\ 4) = p_a$, giving the entire list.

The rest is a counting argument, organized by whether a value repeats among the three subdiagonal entries. If $\T_{z(j),j} = z$ and the subdiagonal at another column $j_1$ carries $p_a$, then the quotient $z^{-1}p_a = (a\ 4)$ must meet $\{j_1, j\}$ by compatibility with the copy of $z$ in column $j$ (\cref{lem:compat}), and $a \neq j_1$, so $a = j$.

Suppose first that the three subdiagonal entries are pairwise distinct. If $z$ is among them, the other two are distinct perturbations, and pinning forces both of their indices to equal the $z$-column, a contradiction. So all three are perturbations, and they form a directed triangle: since $p_a$ agrees with $z$ off column $a$, the entry $p_{a_j}$ at $(z(j),j)$ has value $z(j') = p_{a_{j'}}(j')$ at each other column $j'$ (as $a_{j'} \neq j'$), so it sits at a table cell of $p_{a_{j'}}$. Cyclically this gives a directed triangle, forbidden since $\T$ is triangle-free.

So some value repeats. If $z$ repeats, pinning leaves no perturbation for the remaining subdiagonal, $z$ occupies all four of its table cells.

Otherwise a perturbation repeats, and its common value is the $p$ indexed by neither of its two columns, so by the residual cyclic symmetry normalize $\T_{z(2),2} = \T_{z(3),3} = p_1$, the third subdiagonal carrying $z$, $p_2$ or $p_3$. Composing transpose-inverse with conjugation by $(1\,2)$ restores the normalization, fixes $e$, $z$ and $p_1$, swaps $p_2$ with $p_3$, and swaps the subdiagonal columns $2$ and $3$, so we may assume the third subdiagonal is $z$ or $p_2$. In column $3$ the left quotient of $e$ at $(3,3)$ with $p_1$ at $(1,3)$ is the $4$-cycle $(3\,2\,4\,1)$, so by the Path Metric (\cref{lem:pathmetric}) the column path is $3\!-\!2\!-\!4\!-\!1$ and $\T_{2,3} = (2\,3)\,e$, whose quotient with $p_2$ at $(2,1)$ would be $(1\,3)(2\,4)$, barred from a shared row by the Line Lemma (\cref{cor:linelemma}): the third subdiagonal is $z$.

The subdiagonals are thus $z, p_1, p_1$ at columns $1, 2, 3$, and now we apply the Path Metric (\cref{lem:pathmetric}) to determine the table. In row $1$ the right quotient of $e$ at $(1,1)$ with $p_1$ at $(1,3)$ is the $4$-cycle $(1\,4\,2\,3)$: the row path is $1\!-\!4\!-\!2\!-\!3$ and $\T_{1,4} = e\,(1\,4)$. In column $2$ the left quotient of $e$ at $(2,2)$ with $p_1$ at $(3,2)$ is the $4$-cycle $(2\,4\,1\,3)$: the column path is $2\!-\!4\!-\!1\!-\!3$ and $\T_{4,2} = (2\,4)\,e$. In row $4$ the right quotient of $(2\,4)\,e$ with $z$ is the $4$-cycle $(2\,3\,1\,4)$: the row path is $2\!-\!3\!-\!1\!-\!4$, and $\T_{4,1} = z\,(1\,4) = p_1$. In column $4$ the left quotient of $z$ with $e\,(1\,4)$ at $(1,4)$ is the $4$-cycle $(4\,2\,3\,1)$: the column path is $4\!-\!2\!-\!3\!-\!1$, the neighbour of $z$ is row $2$, and $\T_{2,4} = (2\,4)\,z = p_1$. Now $p_1$ occupies its four table cells $(z(2),2), (z(3),3), (4,1), (2,4)$, so it has multiplicity $4$. The four remaining cells are forced as well: the paths of rows $1$ and $4$ determine $\T_{1,2}$ and $\T_{4,3}$, and Completion (\cref{cor:completion}) applied to row $3$ determines $\T_{3,1}$ and $\T_{3,4}$. The resulting ensemble is exactly the running example of \cref{sec:algebra} (\cref{tab:running}, reproduced in \cref{tab:runningpromo}).
\end{proof}

\begin{table}[h]
\centering
$$\begin{matrix}
\uu{1}234 & 4\uu{1}32 & 43\uu{1}2 & 423\uu{1}\\
\uu{2}314 & 1\uu{2}34 & 13\uu{2}4 & 431\uu{2}\\
\uu{3}214 & 4\uu{3}12 & 12\uu{3}4 & 421\uu{3}\\
\uu{4}312 & 1\uu{4}32 & 13\uu{4}2 & 231\uu{4}\\
\end{matrix}$$
\caption{The running example of \cref{tab:running}, reproduced for convenience: the ensemble forced by the promotion argument when a perturbation repeats, with $z = 231\uu{4}$ and $p_1 = 4312$ of multiplicity four.}
\label{tab:runningpromo}
\end{table}

\begin{theorem}
\label{thm:s4forward}
If $\T$ is a triangle-free $S_4$-ensemble, some $w \in S_4$ appears $4$ times. Equivalently: if the blocking graph of an $S_4$-ensemble contains no directed triangle, then it contains a sink.
\end{theorem}
\begin{proof}
By \cref{thm:mult3}, $\T$ contains a permutation of multiplicity at least $3$, by \cref{lem:mult3done}, some permutation has multiplicity $4$. The two formulations agree by \cref{def:pattern} and \cref{prop:blockdeg}(1).
\end{proof}

\subsection{Rigidity: the converse}

We now show that the pattern not only forbids full multiplicity, but freezes the entire ensemble.

\begin{theorem}[Rigidity]
\label{thm:rigid}
An $S_4$-ensemble containing a pattern is unique up to the symmetries of \cref{lem:symmetry}: it is the ensemble $E^*_4$ of \cref{tab:estar4}. In particular:
\begin{enumerate}
\item no permutation appears $4$ times in such an ensemble (the multiplicities of $E^*_4$ are $3,3,3,1,1,\dots,1$),
\item every pattern is contained in exactly one $S_4$-ensemble, and every $S_4$-ensemble without a permutation of multiplicity $4$ contains exactly one pattern. There are exactly $192$ of each.
\end{enumerate}
\end{theorem}
\begin{proof}
By \cref{lem:symmetry} and \cref{lem:patternform} normalize the pattern so that its three entries are the transpositions $\T_{1,1} = (3\,4)$, $\T_{2,2} = (1\,4)$, $\T_{3,3} = (2\,4)$, the diagonal of \cref{tab:estar4}. Two symmetries fix this configuration. One is given by the rotation $\gamma$: the two-sided action of \cref{lem:symmetry}(1) with $\pi = \rho = (1\,2\,3)$, which sends the cell $(i,j)$ to $((1\,2\,3)i, (1\,2\,3)j)$ and conjugates values by $(1\,2\,3)$, hence permutes the three entries cyclically. Another is the transpose-inverse, which fixes each of the three entries, as they are involutions at diagonal cells. By \cref{cor:transport}, each deduction below yields its images under $\gamma$ and transpose-inverse for free.

The row link of $(3\,4)$ at $(1,1)$ is $(3\,4)(1\ c)$ for some $c \in \{2,3,4\}$. At $c = 4$ the candidate is $324\uu1$, whose quotient with $(1\,4)$ at $(2,2)$ is $(1\,3)$, avoiding $\{4,2\}$ and hence incompatible by \cref{lem:compat}. At $c = 2$ the candidate is $2\uu143$, sharing column $2$ with $(1\,4)$, with left quotient the $4$-cycle $(1\,3\,4\,2)$: by the Path Metric (\cref{lem:pathmetric}) column $2$ would be the path $1\!-\!3\!-\!4\!-\!2$, placing $(3\,4)(1\,3)\,2143 = 2\uu431$ at $(4,2)$, whose quotient with $(2\,4)$ at $(3,3)$ is $(1\,4)$, avoiding $\{2,3\}$ and hence incompatible by \cref{lem:compat}. So $\T_{1,3} = (3\,4)(1\,3) = 42\uu13$, and column $3$, now containing $42\uu13$ and $(2\,4)$ with left quotient the $4$-cycle $(1\,4\,2\,3)$, is the path $1\!-\!4\!-\!2\!-\!3$, placing $(1\,4)\,4213 = 12\uu43$ at $(4,3)$ and $(2\,4)\,1243 = 14\uu23$ at $(2,3)$. Transport under $\gamma$ and $\gamma^2$ (\cref{cor:transport}), applied to these three derived entries, determines columns $1$ and $2$ as in \cref{tab:estar4}.
\begin{table}[h]
\centering
$$\begin{matrix}
\centering
\uu{1}243 & 4\uu{1}32 & 42\uu{1}3 & 423\uu{1}\\
\uu{2}431 & 4\uu{2}31 & 14\uu{2}3 & 143\uu{2}\\
\uu{3}241 & 1\uu{3}42 & 14\uu{3}2 & 124\uu{3}\\
\uu{4}231 & 1\uu{4}32 & 12\uu{4}3 & 123\uu{4}\\
\end{matrix}$$
\caption{The rigid ensemble $E^*_4$. In cycle notation, $\T_{i,j} = (i\ 4)(j\ 4)$ for $i \neq j$ (with the convention $(4\ 4) = e$), $\T_{4,4} = e$, and the diagonal carries the pattern: $\T_{i,i} = (i{-}1\ \ 4)$ for $i \in [3]$, indices cyclic in $[3]$.}
\label{tab:estar4}
\end{table}

In row $4$ the three placed entries $\uu4231, 1\uu432, 12\uu43$ have pairwise $3$-cycle quotients, so all three pairs are at distance two in the link graph, which a path on four vertices cannot host: row $4$ is the star centered at column $4$, and its center is $\T_{4,4} = 4231\,(1\,4) = 123\uu4 = e$. Finally, transport under transpose-inverse (\cref{cor:transport}) gives $\T_{c,4} = \T_{4,c}^{-1}$, determining column $4$ as in \cref{tab:estar4}.
Every cell is thus determined, and the table is exactly the ensemble $E^*_4$ of \cref{tab:estar4}. One reads off that the three transpositions $(1\,4),(2\,4),(3\,4)$ each occur three times and no permutation occurs four times, giving claim (1).

For the counting in claim (2): each of the $192$ patterns extends, by the argument just given, to a unique ensemble, and these ensembles form the single symmetry orbit of $E^*_4$ (whose stabilizer has order $6$, so the orbit has size $1152/6 = 192$). That every ensemble \emph{without} a multiplicity-four permutation contains a pattern follows from \cref{thm:s4forward}. Each such ensemble equals $E^*_4$ up to symmetry by the rigidity just proved, so the $S_4$-ensembles without a multiplicity-four permutation are exactly the $192$-element orbit of $E^*_4$, in bijection with the patterns.
\end{proof}

\begin{theorem}
\label{thm:s4ensemble}
An $S_4$-ensemble $\T$ has a permutation of multiplicity four (equivalently, its blocking graph has a sink) if and only if $\T$ is triangle-free.
\end{theorem}
\begin{proof}
Triangle-free implying full multiplicity is \cref{thm:s4forward}. Conversely, if $\T$ contained a pattern then by \cref{thm:rigid} it would be the ensemble $E^*_4$, whose maximum multiplicity is three, so a permutation of multiplicity four forces triangle-freeness. 
\end{proof}

\begin{corollary}
\label{cor:4d3mixed}
Given $4\Delta_3$, assume that we have a fine mixed subdivision structure on the four boundary facets that comes from an acyclic system of permutations on the edges. Then we can complete this boundary information into a fine mixed subdivision of $4\Delta_3$ if and only if linkage holds and there are no patterns. Moreover, by \cref{thm:rigid}, the non-completable boundaries satisfying linkage are, at the level of their tables, a single orbit under relabeling.
\end{corollary}
\begin{proof}
Form the table $\T$ whose row $i$ lists the maximal matchings of the facet opposite $i'$, as in \cref{thm:boundary}. For $n = 4$ two facets of $4\Delta_3$ meet in an edge, so the four facet subdivisions agree on common faces exactly when they induce a common system of permutations on the edges, which is our hypothesis. By the discussion following \cref{def:boundaryME}, the facet data then assembles into a boundary matching ensemble precisely when right linkage holds among the maximal matchings, and on $\T$ this is the column half of the linkage condition, the row half being automatic from left linkage within each facet. When linkage holds, \cref{thm:boundary} makes $\T$ an $S_4$-ensemble whose boundary extends to a fine mixed subdivision of $4\Delta_3$ if and only if some permutation appears four times, and by \cref{thm:s4ensemble} this happens if and only if there is no pattern. The final claim restates \cref{thm:rigid}.
\end{proof}

\section{Further directions}
\label{sec:further}

We begin with the picture at $n = 5$, where the obstruction $E^*_4$ extends naturally. A computer search, carried out with the assistance of Claude (Anthropic), over $S_5$-ensembles with no permutation of multiplicity five shows that, up to symmetry (\cref{lem:symmetry}), there is exactly one: the ensemble $E^*_5$ of \cref{tab:e5}, whose most occurring permutations are $s = (1\,2)$, $c = (1\,2\,3\,4\,5)$, and $t = (1\,4\,5)$, of multiplicities $4$, $4$, $3$. The obstruction is a directed triangle of the blocking graph among these three, each blocking the next from full multiplicity. The same search confirms that $E^*_5$ is the table of a genuine boundary matching ensemble of $5\Delta_4$, so it produces a fine mixed subdivision structure on the boundary of $5\Delta_4$ that cannot be completed to the interior. We state these computational findings without proof. 

\begin{table}[h]
\centering
$$\begin{matrix}
\uu{1}2345 & 2\uu{1}345 & 23\uu{1}45 & 234\uu{1}5 & 2345\uu{1}\\
\uu{2}1345 & 4\uu{2}351 & 43\uu{2}51 & 413\uu{2}5 & 4135\uu{2}\\
\uu{3}2451 & 2\uu{3}451 & 42\uu{3}51 & 214\uu{3}5 & 2145\uu{3}\\
\uu{4}2351 & 2\uu{4}351 & 23\uu{4}51 & 213\uu{4}5 & 2135\uu{4}\\
\uu{5}2341 & 2\uu{5}341 & 23\uu{5}41 & 234\uu{5}1 & 2134\uu{5}\\
\end{matrix}$$
\caption{The ensemble $E^*_5$: the unique $S_5$-ensemble, up to symmetry, in which no permutation appears five times.}
\label{tab:e5}
\end{table}

We now propose the generalization of our multiplicity theorem and promotion theorem.

\begin{conjecture}
\label{conj:multn1}
Every $S_n$-ensemble contains a permutation of multiplicity at least $n-1$. Every triangle-free $S_n$-ensemble has a permutation of multiplicity $n$.
\end{conjecture}

The computer classification behind \cref{tab:e5} verifies both statements for $n = 5$: an $S_5$-ensemble without a permutation of multiplicity five lies in the symmetry orbit of $E^*_5$, which attains multiplicity four and contains a directed triangle.

From \cref{cor:4d3mixed}, avoiding patterns is a necessary condition on the boundary of $n\Delta_3$ for extendability, so we conjecture:

\begin{conjecture}
\label{conj:3dimconj}
Given $n\Delta_3$, assume that we have a fine mixed subdivision structure on the four boundary facets that comes from an acyclic system of permutations on the edges. Then we can complete this boundary information into a fine mixed subdivision of $n\Delta_3$ if and only if linkage holds and there are no patterns.
\end{conjecture}

The main obstacle is that although \cref{thm:s4ensemble} constructs the individual $4$-by-$4$ matchings compatible with the boundary information, we do not know that they satisfy linkage among each other. \cref{conj:3dimconj} is equivalent to a local-to-global principle: the boundary of $n\Delta_3$ is extendable if and only if its restriction to every sub-$4\Delta_3$ is extendable. 

\begin{question}
\label{ques:localglobal}
Is the boundary of a fine mixed subdivision of $m\Delta_{n-1}$ extendable if and only if its restriction to every sub-$n\Delta_{n-1}$ is extendable?
\end{question}

\begin{remark}[The extremal case]
\label{rem:extremal}
For $m \leq n-1$, every valid boundary of $m\Delta_{n-1}$ is extendable: every matching required in a matching ensemble between $[m]$ and $[n]$ has size at most $m \leq n-1$, so its right vertex set is a proper subset of $[n]$ and it already appears in the boundary data of some facet. The boundary determines the entire ensemble, and $m = n$ is the first case where a genuinely new interior matching must be produced.
\end{remark}

Finally, the geometry behind the obstructions remains mysterious.

\begin{question}
What is the geometric or topological interpretation of the pattern, the blocking cycles, and the ensembles $E^*_4$ and $E^*_5$, in terms of tropical hyperplane arrangements, tropical oriented matroids \cite{ArdilaDevelin07}, or trianguloids \cite{GalashinNenashevPostnikov23}? The unified axiomatic framework of \cite{Yao25} may be the natural setting for such an interpretation.
\end{question}

\subsection*{Acknowledgements}
The author thanks Thalia Kahozi, Alex Kim, and Katherine Liu for useful discussions that made the finding of the pattern possible. Claude (Anthropic) and Gemini (Google) were used to generate the figures, examples, tables of this paper and for editing the prose. Moreover Claude (Anthropic) was used to check every step, especially the ones that are an exhaustive search, for every proof in Section 4.

\printbibliography
\end{document}

%% file: tikz/3d2trilb.tex
\begin{tikzpicture}[scale=2]
\coordinate (A) at (0,0);                
\coordinate (B) at (3,0);                
\coordinate (C) at (1.5,{3*sqrt(3)/2});  

\foreach \i in {0,...,3}{\foreach \j in {0,...,3}{
  \pgfmathtruncatemacro{\k}{3-\i-\j}
  \ifnum\k<0\else
    \path let \p1=(A),\p2=(B),\p3=(C),
      \n1={(\i*\x1+\j*\x2+\k*\x3)/3},
      \n2={(\i*\y1+\j*\y2+\k*\y3)/3}
    in coordinate (P\i\j) at (\n1,\n2);
  \fi}}

\fill[red!12]   (P30) -- (P20) -- (P21) -- cycle;   
\fill[blue!12]  (P20) -- (P10) -- (P11) -- cycle;   
\fill[green!14] (P03) -- (P02) -- (P12) -- cycle;   

\draw[very thick] (A) -- (B) -- (C) -- cycle;

\draw (P20) -- (P21);
\draw (P02) -- (P12);
\draw (P10) -- (P11);
\draw (P20) -- (P11);
\draw (P11) -- (P01);
\draw (P12) -- (P11);

\foreach \i in {0,...,3}{\foreach \j in {0,...,3}{
  \pgfmathtruncatemacro{\k}{3-\i-\j}\ifnum\k<0\else
    \fill (P\i\j) circle (0.9pt);
  \fi}}

\node[red!70!black]   at (barycentric cs:P30=1,P20=1,P21=1) {\small $1$};
\node[blue!70!black]  at (barycentric cs:P20=1,P10=1,P11=1) {\small $2$};
\node[green!55!black] at (barycentric cs:P03=1,P02=1,P12=1) {\small $3$};

\node[below left, xshift=-1pt] at (A) {$1'$};
\node[below right, xshift=1pt] at (B) {$3'$};
\node[above, yshift=1pt]       at (C) {$2'$};
\end{tikzpicture}

%% file: tikz/3d2trivoro.tex
\begin{tikzpicture}[scale=2,
  vor/.style={line width=1.1pt, line cap=round, line join=round},
  vlab/.style={circle, fill=white, inner sep=0.4pt, font=\footnotesize}]
\coordinate (A) at (0,0);                
\coordinate (B) at (3,0);                
\coordinate (C) at (1.5,{3*sqrt(3)/2});  

\foreach \i in {0,...,3}{\foreach \j in {0,...,3}{
  \pgfmathtruncatemacro{\k}{3-\i-\j}
  \ifnum\k<0\else
    \path let \p1=(A),\p2=(B),\p3=(C),
      \n1={(\i*\x1+\j*\x2+\k*\x3)/3},
      \n2={(\i*\y1+\j*\y2+\k*\y3)/3}
    in coordinate (P\i\j) at (\n1,\n2);
  \fi}}

\draw[gray!55, thin] (P20)--(P21);
\draw[gray!55, thin] (P02)--(P12);
\draw[gray!55, thin] (P10)--(P11);
\draw[gray!55, thin] (P20)--(P11);
\draw[gray!55, thin] (P11)--(P01);
\draw[gray!55, thin] (P12)--(P11);
\draw[gray!70, thin] (A)--(B)--(C)--cycle;

\coordinate (Q00) at ($(P00)!0.333!(P10)!0.333!(P01)$);
\coordinate (Q10) at ($(P10)!0.333!(P20)!0.333!(P11)$);
\coordinate (Q01) at ($(P01)!0.333!(P11)!0.333!(P02)$);
\coordinate (Q20) at ($(P20)!0.333!(P30)!0.333!(P21)$);
\coordinate (Q11) at ($(P11)!0.333!(P21)!0.333!(P12)$);
\coordinate (Q02) at ($(P02)!0.333!(P12)!0.333!(P03)$);
\coordinate (B120) at ($(P30)!0.5!(P20)$);
\coordinate (B121) at ($(P20)!0.5!(P10)$);
\coordinate (B122) at ($(P10)!0.5!(P00)$);
\coordinate (B230) at ($(P00)!0.5!(P01)$);
\coordinate (B231) at ($(P01)!0.5!(P02)$);
\coordinate (B232) at ($(P02)!0.5!(P03)$);
\coordinate (B130) at ($(P30)!0.5!(P21)$);
\coordinate (B131) at ($(P21)!0.5!(P12)$);
\coordinate (B132) at ($(P12)!0.5!(P03)$);

\draw[red!75!black, vor]   (B120) -- (Q20) -- (Q11) -- (Q01) -- (B231);
\draw[red!75!black, vor]   (Q20) -- (B130);
\draw[blue!70!black, vor]  (B121) -- (Q10) -- (Q00) -- (B230);
\draw[blue!70!black, vor]  (Q10) -- (Q11) -- (B131);
\draw[green!45!black, vor] (B122) -- (Q00) -- (Q01) -- (Q02) -- (B232);
\draw[green!45!black, vor] (Q02) -- (B132);

\foreach \q in {Q00,Q10,Q01,Q20,Q11,Q02}{\fill[black] (\q) circle (0.7pt);}

\node[red!75!black, vlab]   at (B120) {$1$};
\node[red!75!black, vlab]   at (B231) {$1$};
\node[red!75!black, vlab]   at (B130) {$1$};
\node[blue!70!black, vlab]  at (B121) {$2$};
\node[blue!70!black, vlab]  at (B230) {$2$};
\node[blue!70!black, vlab]  at (B131) {$2$};
\node[green!45!black, vlab] at (B122) {$3$};
\node[green!45!black, vlab] at (B232) {$3$};
\node[green!45!black, vlab] at (B132) {$3$};

\node[below left,  xshift=-1pt] at (P30) {$1'$};
\node[above,       yshift=1pt]  at (P00) {$2'$};
\node[below right, xshift=1pt]  at (P03) {$3'$};
\end{tikzpicture}

%% file: tikz/cayley_tikz.tex
\begin{tikzpicture}[
  front/.style={line width=1pt,blue!62!black},
  back/.style={line width=1pt,blue!62!black,opacity=0.4},
  vert/.style={line width=1pt,brown!72!black,opacity=0.75},
  diag/.style={densely dotted,line width=0.9pt,black!65},
  slicel/.style={line width=1pt,green!55!black},
  cell/.style={line width=0.9pt,black!78},
  mk/.style={line width=0.9pt,cyan!58!black,fill=cyan!14},
  mkpt/.style={cyan!62!black},
  plus/.style={font=\small},
  vlab/.style={font=\normalsize\bfseries,blue!62!black},
  clab/.style={font=\scriptsize},
  dot/.style={circle,fill=blue!62!black,inner sep=1.1pt},
  xdot/.style={circle,fill=green!50!black,inner sep=0.9pt}]
\begin{scope}
\draw[back] (1.350,1.150)--(3.350,1.150)--(1.350,3.150)--cycle;
\draw[diag] (0.000,0.000)--(1.350,1.150);
\draw[diag] (2.000,0.000)--(1.350,1.150);
\draw[diag] (0.000,2.000)--(1.350,1.150);
\draw[diag] (2.000,0.000)--(3.350,1.150);
\draw[diag] (0.000,2.000)--(3.350,1.150);
\draw[diag] (0.000,2.000)--(1.350,3.150);
\draw[vert] (0.000,0.000)--(1.350,1.150);
\draw[vert] (2.000,0.000)--(3.350,1.150);
\draw[vert] (0.000,2.000)--(1.350,3.150);
\fill[green!55!black,fill opacity=0.10] (0.675,0.575) -- (1.675,0.575) -- (0.675,1.575) -- cycle;
\fill[green!55!black,fill opacity=0.10] (1.675,0.575) -- (2.675,0.575) -- (1.675,1.575) -- (0.675,1.575) -- cycle;
\fill[green!55!black,fill opacity=0.10] (0.675,1.575) -- (1.675,1.575) -- (0.675,2.575) -- cycle;
\draw[slicel] (0.675,0.575) -- (1.675,0.575) -- (0.675,1.575) -- cycle;
\draw[slicel] (1.675,0.575) -- (2.675,0.575) -- (1.675,1.575) -- (0.675,1.575) -- cycle;
\draw[slicel] (0.675,1.575) -- (1.675,1.575) -- (0.675,2.575) -- cycle;
\node[xdot] at (0.675,0.575) {};
\node[xdot] at (1.675,0.575) {};
\node[xdot] at (0.675,1.575) {};
\node[xdot] at (2.675,0.575) {};
\node[xdot] at (1.675,1.575) {};
\node[xdot] at (0.675,2.575) {};
\node[clab,green!35!black] at (1.01,0.91) {$123,1$};
\node[clab,green!35!black] at (1.68,1.07) {$23,12$};
\node[clab,green!35!black] at (1.01,1.91) {$3,123$};
\draw[front] (0.000,0.000)--(2.000,0.000)--(0.000,2.000)--cycle;
\node[dot] at (0.000,0.000) {};
\node[dot,opacity=0.55] at (1.350,1.150) {};
\node[dot] at (2.000,0.000) {};
\node[dot,opacity=0.55] at (3.350,1.150) {};
\node[dot] at (0.000,2.000) {};
\node[dot,opacity=0.55] at (1.350,3.150) {};
\node[vlab,below left] at (0.000,0.000) {$1$};
\node[vlab,below right] at (2.000,0.000) {$2$};
\node[vlab,above left] at (0.000,2.000) {$3$};
\node[clab] at (1.8,-0.7) {$\Delta^1\times\Delta^2$ triangulated};
\end{scope}
\begin{scope}[shift={(5.0,-0.2)},scale=1.5]
  \coordinate (V1) at (0,0); \coordinate (V2) at (2,0); \coordinate (V3) at (0,2);
  \fill[black!4] (0,0)--(1,0)--(0,1)--cycle;
  \fill[black!4] (0,1)--(1,0)--(2,0)--(1,1)--cycle;
  \fill[black!4] (0,1)--(1,1)--(0,2)--cycle;
  \draw[cell] (0,0)--(1,0)--(0,1)--cycle;
  \draw[cell] (0,1)--(1,0)--(2,0)--(1,1)--cycle;
  \draw[cell] (0,1)--(1,1)--(0,2)--cycle;
  \draw[line width=1pt,black] (V1)--(V2)--(V3)--cycle;
  \node[vlab,below left] at (V1) {$1$};
  \node[vlab,below right] at (V2) {$2$};
  \node[vlab,above left] at (V3) {$3$};
  \node[clab] at (0.33,0.33) {$123,1$};
  \node[clab] at (1.00,0.50) {$23,12$};
  \node[clab] at (0.33,1.33) {$3,123$};
\end{scope}
\begin{scope}[shift={(9.3, 2.3)}]
  \node[clab,left] at (0,0.25) {$3,123 =$};
  \fill[mkpt] (0.2,0.25) circle (1.7pt);
  \node[plus] at (0.55,0.25) {$+$};
  \draw[mk] (0.9,-0.05)--(1.7,-0.05)--(0.9,0.75)--cycle;
\end{scope}
\begin{scope}[shift={(9.3, 0.9)}]
  \node[clab,left] at (0,0.25) {$23,12 =$};
  \draw[mk] (0.2,-0.05)--(0.2,0.75);
  \node[plus] at (0.55,0.25) {$+$};
  \draw[mk] (0.9,0.25)--(1.7,0.25);
\end{scope}
\begin{scope}[shift={(9.3, -0.5)}]
  \node[clab,left] at (0,0.25) {$123,1 =$};
  \draw[mk] (0.2,-0.05)--(1.0,-0.05)--(0.2,0.75)--cycle;
  \node[plus] at (1.35,0.25) {$+$};
  \fill[mkpt] (1.75,0.25) circle (1.7pt);
\end{scope}
\end{tikzpicture}

%% file: tikz/33mixed_tikz.tex
\begin{tikzpicture}[scale=1.25,
  Ldot/.style={circle,fill=blue!55!black,inner sep=1.4pt},
  Rdot/.style={circle,fill=red!60!black,inner sep=1.4pt},
  tedge/.style={line width=0.9pt,black!78},
  arr/.style={-{Stealth[length=4pt]},gray!65,line width=0.7pt,shorten >=3pt,shorten <=3pt},
  clab/.style={font=\scriptsize}]
\fill[black!5] (1.5,0.866) -- (2.0,0.0) -- (3.0,0.0) -- (2.5,0.866) -- cycle;
\fill[black!5] (0.5,0.866) -- (0.0,0.0) -- (1.0,0.0) -- (1.5,0.866) -- cycle;
\fill[blue!9] (1.0,0.0) -- (2.0,0.0) -- (1.5,0.866) -- cycle;
\fill[blue!9] (1.5,0.866) -- (2.5,0.866) -- (2.0,1.732) -- cycle;
\fill[black!5] (1.0,1.732) -- (0.5,0.866) -- (1.5,0.866) -- (2.0,1.732) -- cycle;
\fill[blue!9] (1.0,1.732) -- (2.0,1.732) -- (1.5,2.598) -- cycle;
\draw[black!55,thin] (1.5,0.866) -- (2.0,0.0) -- (3.0,0.0) -- (2.5,0.866) -- cycle;
\draw[black!55,thin] (0.5,0.866) -- (0.0,0.0) -- (1.0,0.0) -- (1.5,0.866) -- cycle;
\draw[black!55,thin] (1.0,0.0) -- (2.0,0.0) -- (1.5,0.866) -- cycle;
\draw[black!55,thin] (1.5,0.866) -- (2.5,0.866) -- (2.0,1.732) -- cycle;
\draw[black!55,thin] (1.0,1.732) -- (0.5,0.866) -- (1.5,0.866) -- (2.0,1.732) -- cycle;
\draw[black!55,thin] (1.0,1.732) -- (2.0,1.732) -- (1.5,2.598) -- cycle;
\draw[black!75,thick] (0,0) -- (3,0) -- (1.5,2.598) -- cycle;
\node[below left] at (0,0) {$1'$};
\node[above] at (1.5,2.598) {$2'$};
\node[below right] at (3,0) {$3'$};
\node[clab] at (2.25,0.43) {$23,13,3$};
\node[clab] at (0.75,0.43) {$12,1,13$};
\node[clab] at (1.50,0.29) {$123,1,3$};
\node[clab] at (2.00,1.15) {$2,123,3$};
\node[clab] at (1.25,1.30) {$2,12,13$};
\node[clab] at (1.50,2.02) {$2,2,123$};
\begin{scope}[shift={(3.5,-0.5)}]
  \coordinate (G0) at (0.42,0.5);
  \node[Ldot] (G0L1) at (0,1.0) {};
  \node[Rdot] (G0R1) at (0.85,1.0) {};
  \node[Ldot] (G0L2) at (0,0.5) {};
  \node[Rdot] (G0R2) at (0.85,0.5) {};
  \node[Ldot] (G0L3) at (0,0.0) {};
  \node[Rdot] (G0R3) at (0.85,0.0) {};
  \node[font=\tiny,left=0pt] at (G0L1) {$1$};
  \node[font=\tiny,left=0pt] at (G0L2) {$2$};
  \node[font=\tiny,left=0pt] at (G0L3) {$3$};
  \node[font=\tiny,right=0pt] at (G0R1) {$1'$};
  \node[font=\tiny,right=0pt] at (G0R2) {$2'$};
  \node[font=\tiny,right=0pt] at (G0R3) {$3'$};
  \draw[tedge] (G0L1) -- (G0R2);
  \draw[tedge] (G0L1) -- (G0R3);
  \draw[tedge] (G0L2) -- (G0R1);
  \draw[tedge] (G0L2) -- (G0R3);
  \draw[tedge] (G0L3) -- (G0R3);
\end{scope}
\begin{scope}[shift={(-2.7,-1.2)}]
  \coordinate (G1) at (0.42,0.5);
  \node[Ldot] (G1L1) at (0,1.0) {};
  \node[Rdot] (G1R1) at (0.85,1.0) {};
  \node[Ldot] (G1L2) at (0,0.5) {};
  \node[Rdot] (G1R2) at (0.85,0.5) {};
  \node[Ldot] (G1L3) at (0,0.0) {};
  \node[Rdot] (G1R3) at (0.85,0.0) {};
  \node[font=\tiny,left=0pt] at (G1L1) {$1$};
  \node[font=\tiny,left=0pt] at (G1L2) {$2$};
  \node[font=\tiny,left=0pt] at (G1L3) {$3$};
  \node[font=\tiny,right=0pt] at (G1R1) {$1'$};
  \node[font=\tiny,right=0pt] at (G1R2) {$2'$};
  \node[font=\tiny,right=0pt] at (G1R3) {$3'$};
  \draw[tedge] (G1L1) -- (G1R1);
  \draw[tedge] (G1L1) -- (G1R2);
  \draw[tedge] (G1L2) -- (G1R1);
  \draw[tedge] (G1L3) -- (G1R1);
  \draw[tedge] (G1L3) -- (G1R3);
\end{scope}
\begin{scope}[shift={(0.1,-2.3)}]
  \coordinate (G2) at (0.42,0.5);
  \node[Ldot] (G2L1) at (0,1.0) {};
  \node[Rdot] (G2R1) at (0.85,1.0) {};
  \node[Ldot] (G2L2) at (0,0.5) {};
  \node[Rdot] (G2R2) at (0.85,0.5) {};
  \node[Ldot] (G2L3) at (0,0.0) {};
  \node[Rdot] (G2R3) at (0.85,0.0) {};
  \node[font=\tiny,left=0pt] at (G2L1) {$1$};
  \node[font=\tiny,left=0pt] at (G2L2) {$2$};
  \node[font=\tiny,left=0pt] at (G2L3) {$3$};
  \node[font=\tiny,right=0pt] at (G2R1) {$1'$};
  \node[font=\tiny,right=0pt] at (G2R2) {$2'$};
  \node[font=\tiny,right=0pt] at (G2R3) {$3'$};
  \draw[tedge] (G2L1) -- (G2R1);
  \draw[tedge] (G2L1) -- (G2R2);
  \draw[tedge] (G2L1) -- (G2R3);
  \draw[tedge] (G2L2) -- (G2R1);
  \draw[tedge] (G2L3) -- (G2R3);
\end{scope}
\begin{scope}[shift={(3.6,1.7)}]
  \coordinate (G3) at (0.42,0.5);
  \node[Ldot] (G3L1) at (0,1.0) {};
  \node[Rdot] (G3R1) at (0.85,1.0) {};
  \node[Ldot] (G3L2) at (0,0.5) {};
  \node[Rdot] (G3R2) at (0.85,0.5) {};
  \node[Ldot] (G3L3) at (0,0.0) {};
  \node[Rdot] (G3R3) at (0.85,0.0) {};
  \node[font=\tiny,left=0pt] at (G3L1) {$1$};
  \node[font=\tiny,left=0pt] at (G3L2) {$2$};
  \node[font=\tiny,left=0pt] at (G3L3) {$3$};
  \node[font=\tiny,right=0pt] at (G3R1) {$1'$};
  \node[font=\tiny,right=0pt] at (G3R2) {$2'$};
  \node[font=\tiny,right=0pt] at (G3R3) {$3'$};
  \draw[tedge] (G3L1) -- (G3R2);
  \draw[tedge] (G3L2) -- (G3R1);
  \draw[tedge] (G3L2) -- (G3R2);
  \draw[tedge] (G3L2) -- (G3R3);
  \draw[tedge] (G3L3) -- (G3R3);
\end{scope}
\begin{scope}[shift={(-3.0,1.0)}]
  \coordinate (G4) at (0.42,0.5);
  \node[Ldot] (G4L1) at (0,1.0) {};
  \node[Rdot] (G4R1) at (0.85,1.0) {};
  \node[Ldot] (G4L2) at (0,0.5) {};
  \node[Rdot] (G4R2) at (0.85,0.5) {};
  \node[Ldot] (G4L3) at (0,0.0) {};
  \node[Rdot] (G4R3) at (0.85,0.0) {};
  \node[font=\tiny,left=0pt] at (G4L1) {$1$};
  \node[font=\tiny,left=0pt] at (G4L2) {$2$};
  \node[font=\tiny,left=0pt] at (G4L3) {$3$};
  \node[font=\tiny,right=0pt] at (G4R1) {$1'$};
  \node[font=\tiny,right=0pt] at (G4R2) {$2'$};
  \node[font=\tiny,right=0pt] at (G4R3) {$3'$};
  \draw[tedge] (G4L1) -- (G4R2);
  \draw[tedge] (G4L2) -- (G4R1);
  \draw[tedge] (G4L2) -- (G4R2);
  \draw[tedge] (G4L3) -- (G4R1);
  \draw[tedge] (G4L3) -- (G4R3);
\end{scope}
\begin{scope}[shift={(0.7,3.3)}]
  \coordinate (G5) at (0.42,0.5);
  \node[Ldot] (G5L1) at (0,1.0) {};
  \node[Rdot] (G5R1) at (0.85,1.0) {};
  \node[Ldot] (G5L2) at (0,0.5) {};
  \node[Rdot] (G5R2) at (0.85,0.5) {};
  \node[Ldot] (G5L3) at (0,0.0) {};
  \node[Rdot] (G5R3) at (0.85,0.0) {};
  \node[font=\tiny,left=0pt] at (G5L1) {$1$};
  \node[font=\tiny,left=0pt] at (G5L2) {$2$};
  \node[font=\tiny,left=0pt] at (G5L3) {$3$};
  \node[font=\tiny,right=0pt] at (G5R1) {$1'$};
  \node[font=\tiny,right=0pt] at (G5R2) {$2'$};
  \node[font=\tiny,right=0pt] at (G5R3) {$3'$};
  \draw[tedge] (G5L1) -- (G5R2);
  \draw[tedge] (G5L2) -- (G5R2);
  \draw[tedge] (G5L3) -- (G5R1);
  \draw[tedge] (G5L3) -- (G5R2);
  \draw[tedge] (G5L3) -- (G5R3);
\end{scope}
\draw[arr] (G0) to[bend left=6] (2.25,0.43);
\draw[arr] (G1) to[bend left=6] (0.75,0.43);
\draw[arr] (G2) to[bend left=6] (1.50,0.29);
\draw[arr] (G3) to[bend left=6] (2.00,1.15);
\draw[arr] (G4) to[bend left=6] (1.25,1.30);
\draw[arr] (G5) to[bend left=6] (1.50,2.02);
\end{tikzpicture}

%% file: tikz/matchings_tikz.tex
\begin{tikzpicture}[
  mdot/.style={circle,fill=black,inner sep=0.9pt},
  medge/.style={line width=0.7pt,black!80},
  hd/.style={font=\footnotesize\bfseries},
  ax/.style={font=\scriptsize}]
\node[hd] at (1.0,0.9) {size 1};
\node[ax] at (0.55,0.5) {$B=1$};
\node[ax] at (1.55,0.5) {$B=2$};
\node[ax] at (2.55,0.5) {$B=3$};
\node[ax] at (-0.55,-0.15) {$A=1$};
\begin{scope}[shift={(0.30,0.00)}]
  \node[mdot] (S100a1) at (0.00,0.00) {};
  \node[mdot] (S100b1) at (0.62,0.00) {};
  \draw[medge] (S100a1) -- (S100b1);
\end{scope}
\begin{scope}[shift={(1.30,0.00)}]
  \node[mdot] (S101a1) at (0.00,0.00) {};
  \node[mdot] (S101b2) at (0.62,0.00) {};
  \draw[medge] (S101a1) -- (S101b2);
\end{scope}
\begin{scope}[shift={(2.30,0.00)}]
  \node[mdot] (S102a1) at (0.00,0.00) {};
  \node[mdot] (S102b3) at (0.62,0.00) {};
  \draw[medge] (S102a1) -- (S102b3);
\end{scope}
\node[ax] at (-0.55,-1.0999999999999999) {$A=2$};
\begin{scope}[shift={(0.30,-0.95)}]
  \node[mdot] (S110a2) at (0.00,0.00) {};
  \node[mdot] (S110b1) at (0.62,0.00) {};
  \draw[medge] (S110a2) -- (S110b1);
\end{scope}
\begin{scope}[shift={(1.30,-0.95)}]
  \node[mdot] (S111a2) at (0.00,0.00) {};
  \node[mdot] (S111b2) at (0.62,0.00) {};
  \draw[medge] (S111a2) -- (S111b2);
\end{scope}
\begin{scope}[shift={(2.30,-0.95)}]
  \node[mdot] (S112a2) at (0.00,0.00) {};
  \node[mdot] (S112b3) at (0.62,0.00) {};
  \draw[medge] (S112a2) -- (S112b3);
\end{scope}
\node[ax] at (-0.55,-2.05) {$A=3$};
\begin{scope}[shift={(0.30,-1.90)}]
  \node[mdot] (S120a3) at (0.00,0.00) {};
  \node[mdot] (S120b1) at (0.62,0.00) {};
  \draw[medge] (S120a3) -- (S120b1);
\end{scope}
\begin{scope}[shift={(1.30,-1.90)}]
  \node[mdot] (S121a3) at (0.00,0.00) {};
  \node[mdot] (S121b2) at (0.62,0.00) {};
  \draw[medge] (S121a3) -- (S121b2);
\end{scope}
\begin{scope}[shift={(2.30,-1.90)}]
  \node[mdot] (S122a3) at (0.00,0.00) {};
  \node[mdot] (S122b3) at (0.62,0.00) {};
  \draw[medge] (S122a3) -- (S122b3);
\end{scope}
\node[hd] at (5.1,0.9) {size 2};
\node[ax] at (4.55,0.5) {$B=12$};
\node[ax] at (5.55,0.5) {$B=13$};
\node[ax] at (6.55,0.5) {$B=23$};
\node[ax] at (3.4,-0.15) {$A=12$};
\begin{scope}[shift={(4.30,0.00)}]
  \node[mdot] (S200a1) at (0.00,0.00) {};
  \node[mdot] (S200a2) at (0.00,-0.32) {};
  \node[mdot] (S200b1) at (0.62,0.00) {};
  \node[mdot] (S200b2) at (0.62,-0.32) {};
  \draw[medge] (S200a1) -- (S200b2);
  \draw[medge] (S200a2) -- (S200b1);
\end{scope}
\begin{scope}[shift={(5.30,0.00)}]
  \node[mdot] (S201a1) at (0.00,0.00) {};
  \node[mdot] (S201a2) at (0.00,-0.32) {};
  \node[mdot] (S201b1) at (0.62,0.00) {};
  \node[mdot] (S201b3) at (0.62,-0.32) {};
  \draw[medge] (S201a1) -- (S201b3);
  \draw[medge] (S201a2) -- (S201b1);
\end{scope}
\begin{scope}[shift={(6.30,0.00)}]
  \node[mdot] (S202a1) at (0.00,0.00) {};
  \node[mdot] (S202a2) at (0.00,-0.32) {};
  \node[mdot] (S202b2) at (0.62,0.00) {};
  \node[mdot] (S202b3) at (0.62,-0.32) {};
  \draw[medge] (S202a1) -- (S202b2);
  \draw[medge] (S202a2) -- (S202b3);
\end{scope}
\node[ax] at (3.4,-1.0999999999999999) {$A=13$};
\begin{scope}[shift={(4.30,-0.95)}]
  \node[mdot] (S210a1) at (0.00,0.00) {};
  \node[mdot] (S210a3) at (0.00,-0.32) {};
  \node[mdot] (S210b1) at (0.62,0.00) {};
  \node[mdot] (S210b2) at (0.62,-0.32) {};
  \draw[medge] (S210a1) -- (S210b2);
  \draw[medge] (S210a3) -- (S210b1);
\end{scope}
\begin{scope}[shift={(5.30,-0.95)}]
  \node[mdot] (S211a1) at (0.00,0.00) {};
  \node[mdot] (S211a3) at (0.00,-0.32) {};
  \node[mdot] (S211b1) at (0.62,0.00) {};
  \node[mdot] (S211b3) at (0.62,-0.32) {};
  \draw[medge] (S211a1) -- (S211b1);
  \draw[medge] (S211a3) -- (S211b3);
\end{scope}
\begin{scope}[shift={(6.30,-0.95)}]
  \node[mdot] (S212a1) at (0.00,0.00) {};
  \node[mdot] (S212a3) at (0.00,-0.32) {};
  \node[mdot] (S212b2) at (0.62,0.00) {};
  \node[mdot] (S212b3) at (0.62,-0.32) {};
  \draw[medge] (S212a1) -- (S212b2);
  \draw[medge] (S212a3) -- (S212b3);
\end{scope}
\node[ax] at (3.4,-2.05) {$A=23$};
\begin{scope}[shift={(4.30,-1.90)}]
  \node[mdot] (S220a2) at (0.00,0.00) {};
  \node[mdot] (S220a3) at (0.00,-0.32) {};
  \node[mdot] (S220b1) at (0.62,0.00) {};
  \node[mdot] (S220b2) at (0.62,-0.32) {};
  \draw[medge] (S220a2) -- (S220b2);
  \draw[medge] (S220a3) -- (S220b1);
\end{scope}
\begin{scope}[shift={(5.30,-1.90)}]
  \node[mdot] (S221a2) at (0.00,0.00) {};
  \node[mdot] (S221a3) at (0.00,-0.32) {};
  \node[mdot] (S221b1) at (0.62,0.00) {};
  \node[mdot] (S221b3) at (0.62,-0.32) {};
  \draw[medge] (S221a2) -- (S221b1);
  \draw[medge] (S221a3) -- (S221b3);
\end{scope}
\begin{scope}[shift={(6.30,-1.90)}]
  \node[mdot] (S222a2) at (0.00,0.00) {};
  \node[mdot] (S222a3) at (0.00,-0.32) {};
  \node[mdot] (S222b2) at (0.62,0.00) {};
  \node[mdot] (S222b3) at (0.62,-0.32) {};
  \draw[medge] (S222a2) -- (S222b2);
  \draw[medge] (S222a3) -- (S222b3);
\end{scope}
\node[hd] at (8.7,0.9) {size 3};
\node[ax] at (8.5,0.45) {$A=B=123$};
\begin{scope}[shift={(8.50,0.00)}]
  \node[mdot] (S3a1) at (0.00,0.00) {};
  \node[mdot] (S3a2) at (0.00,-0.34) {};
  \node[mdot] (S3a3) at (0.00,-0.68) {};
  \node[mdot] (S3b1) at (0.62,0.00) {};
  \node[mdot] (S3b2) at (0.62,-0.34) {};
  \node[mdot] (S3b3) at (0.62,-0.68) {};
  \draw[medge] (S3a1) -- (S3b2);
  \draw[medge] (S3a2) -- (S3b1);
  \draw[medge] (S3a3) -- (S3b3);
\end{scope}
\node[ax,align=center] at (8.7,-1.3) {interior cell\\ $w=213$};
\end{tikzpicture}

%% file: tikz/staircase4.tex
\begin{tikzpicture}[scale=1.05]
\fill[red!14] (0.0000,0.0000) -- (1.5000,0.0000) -- (0.8250,0.4875) -- cycle;
\draw[black!60, line width=0.35pt] (0.0000,0.0000) -- (1.5000,0.0000);
\draw[black!60, line width=0.35pt] (1.5000,0.0000) -- (0.8250,0.4875);
\draw[black!60, line width=0.35pt] (0.8250,0.4875) -- (0.0000,0.0000);
\draw[black!60, line width=0.35pt] (1.5000,0.0000) -- (3.0000,0.0000);
\draw[black!60, line width=0.35pt] (3.0000,0.0000) -- (2.3250,0.4875);
\draw[black!60, line width=0.35pt] (2.3250,0.4875) -- (0.8250,0.4875);
\draw[black!60, line width=0.35pt] (3.0000,0.0000) -- (4.5000,0.0000);
\draw[black!60, line width=0.35pt] (4.5000,0.0000) -- (3.8250,0.4875);
\draw[black!60, line width=0.35pt] (3.8250,0.4875) -- (2.3250,0.4875);
\draw[black!60, line width=0.35pt] (4.5000,0.0000) -- (6.0000,0.0000);
\draw[black!60, line width=0.35pt] (6.0000,0.0000) -- (5.3250,0.4875);
\draw[black!60, line width=0.35pt] (5.3250,0.4875) -- (3.8250,0.4875);
\fill[blue!14] (0.8250,0.4875) -- (2.3250,0.4875) -- (1.6500,0.9750) -- cycle;
\draw[black!60, line width=0.35pt] (2.3250,0.4875) -- (1.6500,0.9750);
\draw[black!60, line width=0.35pt] (1.6500,0.9750) -- (0.8250,0.4875);
\draw[black!60, line width=0.35pt] (3.8250,0.4875) -- (3.1500,0.9750);
\draw[black!60, line width=0.35pt] (3.1500,0.9750) -- (1.6500,0.9750);
\draw[black!60, line width=0.35pt] (5.3250,0.4875) -- (4.6500,0.9750);
\draw[black!60, line width=0.35pt] (4.6500,0.9750) -- (3.1500,0.9750);
\fill[green!16] (1.6500,0.9750) -- (3.1500,0.9750) -- (2.4750,1.4625) -- cycle;
\draw[black!60, line width=0.35pt] (3.1500,0.9750) -- (2.4750,1.4625);
\draw[black!60, line width=0.35pt] (2.4750,1.4625) -- (1.6500,0.9750);
\draw[black!60, line width=0.35pt] (4.6500,0.9750) -- (3.9750,1.4625);
\draw[black!60, line width=0.35pt] (3.9750,1.4625) -- (2.4750,1.4625);
\fill[orange!22] (2.4750,1.4625) -- (3.9750,1.4625) -- (3.3000,1.9500) -- cycle;
\draw[black!60, line width=0.35pt] (3.9750,1.4625) -- (3.3000,1.9500);
\draw[black!60, line width=0.35pt] (3.3000,1.9500) -- (2.4750,1.4625);
\node[red!70!black, font=\scriptsize] at (0.7750,0.1625) {$1$};
\node[blue!70!black, font=\scriptsize] at (1.6000,0.6500) {$2$};
\node[green!45!black, font=\scriptsize] at (2.4250,1.1375) {$3$};
\node[orange!85!black, font=\scriptsize] at (3.2500,1.6250) {$4$};
\fill[red!14] (5.3250,0.4875) -- (6.0000,0.0000) -- (5.2500,1.2990) -- cycle;
\draw[black!60, line width=0.35pt] (5.3250,0.4875) -- (6.0000,0.0000);
\draw[black!60, line width=0.35pt] (6.0000,0.0000) -- (5.2500,1.2990);
\draw[black!60, line width=0.35pt] (5.2500,1.2990) -- (5.3250,0.4875);
\draw[black!60, line width=0.35pt] (5.2500,1.2990) -- (4.5000,2.5981);
\draw[black!60, line width=0.35pt] (4.5000,2.5981) -- (4.5750,1.7865);
\draw[black!60, line width=0.35pt] (4.5750,1.7865) -- (5.3250,0.4875);
\draw[black!60, line width=0.35pt] (4.5000,2.5981) -- (3.7500,3.8971);
\draw[black!60, line width=0.35pt] (3.7500,3.8971) -- (3.8250,3.0856);
\draw[black!60, line width=0.35pt] (3.8250,3.0856) -- (4.5750,1.7865);
\draw[black!60, line width=0.35pt] (3.7500,3.8971) -- (3.0000,5.1962);
\draw[black!60, line width=0.35pt] (3.0000,5.1962) -- (3.0750,4.3846);
\draw[black!60, line width=0.35pt] (3.0750,4.3846) -- (3.8250,3.0856);
\fill[blue!14] (4.6500,0.9750) -- (5.3250,0.4875) -- (4.5750,1.7865) -- cycle;
\draw[black!60, line width=0.35pt] (4.6500,0.9750) -- (5.3250,0.4875);
\draw[black!60, line width=0.35pt] (4.5750,1.7865) -- (4.6500,0.9750);
\draw[black!60, line width=0.35pt] (3.8250,3.0856) -- (3.9000,2.2740);
\draw[black!60, line width=0.35pt] (3.9000,2.2740) -- (4.6500,0.9750);
\draw[black!60, line width=0.35pt] (3.0750,4.3846) -- (3.1500,3.5731);
\draw[black!60, line width=0.35pt] (3.1500,3.5731) -- (3.9000,2.2740);
\fill[green!16] (3.9750,1.4625) -- (4.6500,0.9750) -- (3.9000,2.2740) -- cycle;
\draw[black!60, line width=0.35pt] (3.9750,1.4625) -- (4.6500,0.9750);
\draw[black!60, line width=0.35pt] (3.9000,2.2740) -- (3.9750,1.4625);
\draw[black!60, line width=0.35pt] (3.1500,3.5731) -- (3.2250,2.7615);
\draw[black!60, line width=0.35pt] (3.2250,2.7615) -- (3.9750,1.4625);
\fill[orange!22] (3.3000,1.9500) -- (3.9750,1.4625) -- (3.2250,2.7615) -- cycle;
\draw[black!60, line width=0.35pt] (3.3000,1.9500) -- (3.9750,1.4625);
\draw[black!60, line width=0.35pt] (3.2250,2.7615) -- (3.3000,1.9500);
\node[red!70!black, font=\scriptsize] at (5.5250,0.5955) {$1$};
\node[blue!70!black, font=\scriptsize] at (4.8500,1.0830) {$2$};
\node[green!45!black, font=\scriptsize] at (4.1750,1.5705) {$3$};
\node[orange!85!black, font=\scriptsize] at (3.5000,2.0580) {$4$};
\fill[red!14] (0.0000,0.0000) -- (0.8250,0.4875) -- (0.7500,1.2990) -- cycle;
\draw[black!60, line width=0.35pt] (0.0000,0.0000) -- (0.8250,0.4875);
\draw[black!60, line width=0.35pt] (0.8250,0.4875) -- (0.7500,1.2990);
\draw[black!60, line width=0.35pt] (0.7500,1.2990) -- (0.0000,0.0000);
\draw[black!60, line width=0.35pt] (0.8250,0.4875) -- (1.5750,1.7865);
\draw[black!60, line width=0.35pt] (1.5750,1.7865) -- (1.5000,2.5981);
\draw[black!60, line width=0.35pt] (1.5000,2.5981) -- (0.7500,1.2990);
\draw[black!60, line width=0.35pt] (1.5750,1.7865) -- (2.3250,3.0856);
\draw[black!60, line width=0.35pt] (2.3250,3.0856) -- (2.2500,3.8971);
\draw[black!60, line width=0.35pt] (2.2500,3.8971) -- (1.5000,2.5981);
\draw[black!60, line width=0.35pt] (2.3250,3.0856) -- (3.0750,4.3846);
\draw[black!60, line width=0.35pt] (3.0750,4.3846) -- (3.0000,5.1962);
\draw[black!60, line width=0.35pt] (3.0000,5.1962) -- (2.2500,3.8971);
\fill[blue!14] (0.8250,0.4875) -- (1.6500,0.9750) -- (1.5750,1.7865) -- cycle;
\draw[black!60, line width=0.35pt] (0.8250,0.4875) -- (1.6500,0.9750);
\draw[black!60, line width=0.35pt] (1.6500,0.9750) -- (1.5750,1.7865);
\draw[black!60, line width=0.35pt] (1.6500,0.9750) -- (2.4000,2.2740);
\draw[black!60, line width=0.35pt] (2.4000,2.2740) -- (2.3250,3.0856);
\draw[black!60, line width=0.35pt] (2.4000,2.2740) -- (3.1500,3.5731);
\draw[black!60, line width=0.35pt] (3.1500,3.5731) -- (3.0750,4.3846);
\fill[green!16] (1.6500,0.9750) -- (2.4750,1.4625) -- (2.4000,2.2740) -- cycle;
\draw[black!60, line width=0.35pt] (1.6500,0.9750) -- (2.4750,1.4625);
\draw[black!60, line width=0.35pt] (2.4750,1.4625) -- (2.4000,2.2740);
\draw[black!60, line width=0.35pt] (2.4750,1.4625) -- (3.2250,2.7615);
\draw[black!60, line width=0.35pt] (3.2250,2.7615) -- (3.1500,3.5731);
\fill[orange!22] (2.4750,1.4625) -- (3.3000,1.9500) -- (3.2250,2.7615) -- cycle;
\draw[black!60, line width=0.35pt] (2.4750,1.4625) -- (3.3000,1.9500);
\draw[black!60, line width=0.35pt] (3.3000,1.9500) -- (3.2250,2.7615);
\node[red!70!black, font=\scriptsize] at (0.5250,0.5955) {$1$};
\node[blue!70!black, font=\scriptsize] at (1.3500,1.0830) {$2$};
\node[green!45!black, font=\scriptsize] at (2.1750,1.5705) {$3$};
\node[orange!85!black, font=\scriptsize] at (3.0000,2.0580) {$4$};
\fill[red!14] (8.2500,3.8971) -- (9.0000,5.1962) -- (7.5000,5.1962) -- cycle;
\draw[black!60, line width=0.35pt] (8.2500,3.8971) -- (9.0000,5.1962);
\draw[black!60, line width=0.35pt] (9.0000,5.1962) -- (7.5000,5.1962);
\draw[black!60, line width=0.35pt] (7.5000,5.1962) -- (8.2500,3.8971);
\draw[black!60, line width=0.35pt] (7.5000,2.5981) -- (8.2500,3.8971);
\draw[black!60, line width=0.35pt] (7.5000,5.1962) -- (6.7500,3.8971);
\draw[black!60, line width=0.35pt] (6.7500,3.8971) -- (7.5000,2.5981);
\draw[black!60, line width=0.35pt] (6.7500,1.2990) -- (7.5000,2.5981);
\draw[black!60, line width=0.35pt] (6.7500,3.8971) -- (6.0000,2.5981);
\draw[black!60, line width=0.35pt] (6.0000,2.5981) -- (6.7500,1.2990);
\draw[black!60, line width=0.35pt] (6.0000,0.0000) -- (6.7500,1.2990);
\draw[black!60, line width=0.35pt] (6.0000,2.5981) -- (5.2500,1.2990);
\draw[black!60, line width=0.35pt] (5.2500,1.2990) -- (6.0000,0.0000);
\fill[blue!14] (6.7500,3.8971) -- (7.5000,5.1962) -- (6.0000,5.1962) -- cycle;
\draw[black!60, line width=0.35pt] (7.5000,5.1962) -- (6.0000,5.1962);
\draw[black!60, line width=0.35pt] (6.0000,5.1962) -- (6.7500,3.8971);
\draw[black!60, line width=0.35pt] (6.0000,5.1962) -- (5.2500,3.8971);
\draw[black!60, line width=0.35pt] (5.2500,3.8971) -- (6.0000,2.5981);
\draw[black!60, line width=0.35pt] (5.2500,3.8971) -- (4.5000,2.5981);
\draw[black!60, line width=0.35pt] (4.5000,2.5981) -- (5.2500,1.2990);
\fill[green!16] (5.2500,3.8971) -- (6.0000,5.1962) -- (4.5000,5.1962) -- cycle;
\draw[black!60, line width=0.35pt] (6.0000,5.1962) -- (4.5000,5.1962);
\draw[black!60, line width=0.35pt] (4.5000,5.1962) -- (5.2500,3.8971);
\draw[black!60, line width=0.35pt] (4.5000,5.1962) -- (3.7500,3.8971);
\draw[black!60, line width=0.35pt] (3.7500,3.8971) -- (4.5000,2.5981);
\fill[orange!22] (3.7500,3.8971) -- (4.5000,5.1962) -- (3.0000,5.1962) -- cycle;
\draw[black!60, line width=0.35pt] (4.5000,5.1962) -- (3.0000,5.1962);
\draw[black!60, line width=0.35pt] (3.0000,5.1962) -- (3.7500,3.8971);
\node[red!70!black, font=\scriptsize] at (8.2500,4.7631) {$1$};
\node[blue!70!black, font=\scriptsize] at (6.7500,4.7631) {$2$};
\node[green!45!black, font=\scriptsize] at (5.2500,4.7631) {$3$};
\node[orange!85!black, font=\scriptsize] at (3.7500,4.7631) {$4$};
\draw[black!80, line width=0.9pt] (0.0000,0.0000) -- (6.0000,0.0000);
\draw[black!80, line width=0.9pt] (6.0000,0.0000) -- (3.0000,5.1962);
\draw[black!80, line width=0.9pt] (3.0000,5.1962) -- (0.0000,0.0000);
\draw[black!80, line width=0.9pt] (0.0000,0.0000) -- (3.3000,1.9500);
\draw[black!80, line width=0.9pt] (6.0000,0.0000) -- (3.3000,1.9500);
\draw[black!80, line width=0.9pt] (3.0000,5.1962) -- (3.3000,1.9500);
\draw[black!80, line width=0.9pt] (9.0000,5.1962) -- (6.0000,0.0000);
\draw[black!80, line width=0.9pt] (9.0000,5.1962) -- (3.0000,5.1962);
\node[red!70!black, circle, fill=white, inner sep=0.3pt, font=\tiny] at (0.7500,0.0000) {$1$};
\node[blue!70!black, circle, fill=white, inner sep=0.3pt, font=\tiny] at (2.2500,0.0000) {$2$};
\node[green!45!black, circle, fill=white, inner sep=0.3pt, font=\tiny] at (3.7500,0.0000) {$3$};
\node[orange!85!black, circle, fill=white, inner sep=0.3pt, font=\tiny] at (5.2500,0.0000) {$4$};
\node[red!70!black, circle, fill=white, inner sep=0.3pt, font=\tiny] at (8.6250,4.5466) {$1$};
\node[blue!70!black, circle, fill=white, inner sep=0.3pt, font=\tiny] at (7.8750,3.2476) {$2$};
\node[green!45!black, circle, fill=white, inner sep=0.3pt, font=\tiny] at (7.1250,1.9486) {$3$};
\node[orange!85!black, circle, fill=white, inner sep=0.3pt, font=\tiny] at (6.3750,0.6495) {$4$};
\node[red!70!black, circle, fill=white, inner sep=0.3pt, font=\tiny] at (0.4125,0.2437) {$1$};
\node[blue!70!black, circle, fill=white, inner sep=0.3pt, font=\tiny] at (1.2375,0.7312) {$2$};
\node[green!45!black, circle, fill=white, inner sep=0.3pt, font=\tiny] at (2.0625,1.2188) {$3$};
\node[orange!85!black, circle, fill=white, inner sep=0.3pt, font=\tiny] at (2.8875,1.7063) {$4$};
\node[red!70!black, circle, fill=white, inner sep=0.3pt, font=\tiny] at (5.6625,0.2437) {$1$};
\node[blue!70!black, circle, fill=white, inner sep=0.3pt, font=\tiny] at (4.9875,0.7312) {$2$};
\node[green!45!black, circle, fill=white, inner sep=0.3pt, font=\tiny] at (4.3125,1.2188) {$3$};
\node[orange!85!black, circle, fill=white, inner sep=0.3pt, font=\tiny] at (3.6375,1.7063) {$4$};
\node[red!70!black, circle, fill=white, inner sep=0.3pt, font=\tiny] at (3.0375,4.7904) {$1$};
\node[blue!70!black, circle, fill=white, inner sep=0.3pt, font=\tiny] at (3.1125,3.9788) {$2$};
\node[green!45!black, circle, fill=white, inner sep=0.3pt, font=\tiny] at (3.1875,3.1673) {$3$};
\node[orange!85!black, circle, fill=white, inner sep=0.3pt, font=\tiny] at (3.2625,2.3558) {$4$};
\node[red!70!black, circle, fill=white, inner sep=0.3pt, font=\tiny] at (0.3750,0.6495) {$1$};
\node[blue!70!black, circle, fill=white, inner sep=0.3pt, font=\tiny] at (1.1250,1.9486) {$2$};
\node[green!45!black, circle, fill=white, inner sep=0.3pt, font=\tiny] at (1.8750,3.2476) {$3$};
\node[orange!85!black, circle, fill=white, inner sep=0.3pt, font=\tiny] at (2.6250,4.5466) {$4$};
\node[red!70!black, circle, fill=white, inner sep=0.3pt, font=\tiny] at (8.2500,5.1962) {$1$};
\node[blue!70!black, circle, fill=white, inner sep=0.3pt, font=\tiny] at (6.7500,5.1962) {$2$};
\node[green!45!black, circle, fill=white, inner sep=0.3pt, font=\tiny] at (5.2500,5.1962) {$3$};
\node[orange!85!black, circle, fill=white, inner sep=0.3pt, font=\tiny] at (3.7500,5.1962) {$4$};
\node[red!70!black, circle, fill=white, inner sep=0.3pt, font=\tiny] at (5.6250,0.6495) {$1$};
\node[blue!70!black, circle, fill=white, inner sep=0.3pt, font=\tiny] at (4.8750,1.9486) {$2$};
\node[green!45!black, circle, fill=white, inner sep=0.3pt, font=\tiny] at (4.1250,3.2476) {$3$};
\node[orange!85!black, circle, fill=white, inner sep=0.3pt, font=\tiny] at (3.3750,4.5466) {$4$};
\node[below left=-1pt] at (0.0000,0.0000) {$1'$};
\node[below right=-1pt] at (6.0000,0.0000) {$2'$};
\node[above=0pt] at (3.0000,5.1962) {$3'$};
\node[circle, fill=white, inner sep=0.5pt, font=\small] at (3.3000,1.6700) {$4'$};
\node[above right=-1pt] at (9.0000,5.1962) {$1'$};
\draw[-{Stealth[length=5pt]}, gray!70, line width=0.8pt] (5.550,3.448) .. controls (5.650,2.748) and (5.050,2.248) .. (4.350,2.048);
\node[gray!70, font=\scriptsize, above right=-2pt] at (5.450,2.748) {fold};
\end{tikzpicture}

%% file: tikz/regular4.tex
\begin{tikzpicture}[scale=1.05]
\fill[red!14] (3.8250,0.4875) -- (5.3250,0.4875) -- (4.6500,0.9750) -- cycle;
\draw[black!60, line width=0.35pt] (3.8250,0.4875) -- (5.3250,0.4875);
\draw[black!60, line width=0.35pt] (5.3250,0.4875) -- (4.6500,0.9750);
\draw[black!60, line width=0.35pt] (4.6500,0.9750) -- (3.8250,0.4875);
\draw[black!60, line width=0.35pt] (4.5000,0.0000) -- (6.0000,0.0000);
\draw[black!60, line width=0.35pt] (6.0000,0.0000) -- (5.3250,0.4875);
\draw[black!60, line width=0.35pt] (3.8250,0.4875) -- (4.5000,0.0000);
\draw[black!60, line width=0.35pt] (2.3250,0.4875) -- (3.8250,0.4875);
\draw[black!60, line width=0.35pt] (4.6500,0.9750) -- (3.1500,0.9750);
\draw[black!60, line width=0.35pt] (3.1500,0.9750) -- (2.3250,0.4875);
\draw[black!60, line width=0.35pt] (0.8250,0.4875) -- (2.3250,0.4875);
\draw[black!60, line width=0.35pt] (3.1500,0.9750) -- (1.6500,0.9750);
\draw[black!60, line width=0.35pt] (1.6500,0.9750) -- (0.8250,0.4875);
\draw[black!60, line width=0.35pt] (1.5000,0.0000) -- (3.0000,0.0000);
\draw[black!60, line width=0.35pt] (3.0000,0.0000) -- (3.8250,0.4875);
\draw[black!60, line width=0.35pt] (2.3250,0.4875) -- (1.5000,0.0000);
\draw[black!60, line width=0.35pt] (0.0000,0.0000) -- (1.5000,0.0000);
\draw[black!60, line width=0.35pt] (0.8250,0.4875) -- (0.0000,0.0000);
\fill[green!16] (3.0000,0.0000) -- (4.5000,0.0000) -- (3.8250,0.4875) -- cycle;
\draw[black!60, line width=0.35pt] (3.0000,0.0000) -- (4.5000,0.0000);
\fill[blue!14] (3.1500,0.9750) -- (4.6500,0.9750) -- (3.9750,1.4625) -- cycle;
\draw[black!60, line width=0.35pt] (4.6500,0.9750) -- (3.9750,1.4625);
\draw[black!60, line width=0.35pt] (3.9750,1.4625) -- (3.1500,0.9750);
\draw[black!60, line width=0.35pt] (3.9750,1.4625) -- (2.4750,1.4625);
\draw[black!60, line width=0.35pt] (2.4750,1.4625) -- (1.6500,0.9750);
\fill[orange!22] (2.4750,1.4625) -- (3.9750,1.4625) -- (3.3000,1.9500) -- cycle;
\draw[black!60, line width=0.35pt] (3.9750,1.4625) -- (3.3000,1.9500);
\draw[black!60, line width=0.35pt] (3.3000,1.9500) -- (2.4750,1.4625);
\node[red!70!black, font=\scriptsize] at (4.6000,0.6500) {$1$};
\node[green!45!black, font=\scriptsize] at (3.7750,0.1625) {$3$};
\node[blue!70!black, font=\scriptsize] at (3.9250,1.1375) {$2$};
\node[orange!85!black, font=\scriptsize] at (3.2500,1.6250) {$4$};
\fill[red!14] (3.9000,2.2740) -- (4.5750,1.7865) -- (3.8250,3.0856) -- cycle;
\draw[black!60, line width=0.35pt] (3.9000,2.2740) -- (4.5750,1.7865);
\draw[black!60, line width=0.35pt] (4.5750,1.7865) -- (3.8250,3.0856);
\draw[black!60, line width=0.35pt] (3.8250,3.0856) -- (3.9000,2.2740);
\draw[black!60, line width=0.35pt] (4.5750,1.7865) -- (5.2500,1.2990);
\draw[black!60, line width=0.35pt] (5.2500,1.2990) -- (4.5000,2.5981);
\draw[black!60, line width=0.35pt] (4.5000,2.5981) -- (3.8250,3.0856);
\draw[black!60, line width=0.35pt] (4.6500,0.9750) -- (5.3250,0.4875);
\draw[black!60, line width=0.35pt] (5.3250,0.4875) -- (4.5750,1.7865);
\draw[black!60, line width=0.35pt] (3.9000,2.2740) -- (4.6500,0.9750);
\draw[black!60, line width=0.35pt] (5.3250,0.4875) -- (6.0000,0.0000);
\draw[black!60, line width=0.35pt] (6.0000,0.0000) -- (5.2500,1.2990);
\draw[black!60, line width=0.35pt] (3.8250,3.0856) -- (3.0750,4.3846);
\draw[black!60, line width=0.35pt] (3.0750,4.3846) -- (3.1500,3.5731);
\draw[black!60, line width=0.35pt] (3.1500,3.5731) -- (3.9000,2.2740);
\fill[green!16] (3.8250,3.0856) -- (4.5000,2.5981) -- (3.7500,3.8971) -- cycle;
\draw[black!60, line width=0.35pt] (4.5000,2.5981) -- (3.7500,3.8971);
\draw[black!60, line width=0.35pt] (3.7500,3.8971) -- (3.8250,3.0856);
\draw[black!60, line width=0.35pt] (3.7500,3.8971) -- (3.0000,5.1962);
\draw[black!60, line width=0.35pt] (3.0000,5.1962) -- (3.0750,4.3846);
\fill[blue!14] (3.9750,1.4625) -- (4.6500,0.9750) -- (3.9000,2.2740) -- cycle;
\draw[black!60, line width=0.35pt] (3.9750,1.4625) -- (4.6500,0.9750);
\draw[black!60, line width=0.35pt] (3.9000,2.2740) -- (3.9750,1.4625);
\draw[black!60, line width=0.35pt] (3.3000,1.9500) -- (3.9750,1.4625);
\draw[black!60, line width=0.35pt] (3.9000,2.2740) -- (3.2250,2.7615);
\draw[black!60, line width=0.35pt] (3.2250,2.7615) -- (3.3000,1.9500);
\fill[orange!22] (3.2250,2.7615) -- (3.9000,2.2740) -- (3.1500,3.5731) -- cycle;
\draw[black!60, line width=0.35pt] (3.1500,3.5731) -- (3.2250,2.7615);
\node[red!70!black, font=\scriptsize] at (4.1000,2.3821) {$1$};
\node[green!45!black, font=\scriptsize] at (4.0250,3.1936) {$3$};
\node[blue!70!black, font=\scriptsize] at (4.1750,1.5705) {$2$};
\node[orange!85!black, font=\scriptsize] at (3.4250,2.8696) {$4$};
\fill[red!14] (1.5750,1.7865) -- (2.4000,2.2740) -- (2.3250,3.0856) -- cycle;
\draw[black!60, line width=0.35pt] (1.5750,1.7865) -- (2.4000,2.2740);
\draw[black!60, line width=0.35pt] (2.4000,2.2740) -- (2.3250,3.0856);
\draw[black!60, line width=0.35pt] (2.3250,3.0856) -- (1.5750,1.7865);
\draw[black!60, line width=0.35pt] (0.7500,1.2990) -- (1.5750,1.7865);
\draw[black!60, line width=0.35pt] (2.3250,3.0856) -- (1.5000,2.5981);
\draw[black!60, line width=0.35pt] (1.5000,2.5981) -- (0.7500,1.2990);
\draw[black!60, line width=0.35pt] (0.8250,0.4875) -- (1.6500,0.9750);
\draw[black!60, line width=0.35pt] (1.6500,0.9750) -- (2.4000,2.2740);
\draw[black!60, line width=0.35pt] (1.5750,1.7865) -- (0.8250,0.4875);
\draw[black!60, line width=0.35pt] (0.0000,0.0000) -- (0.8250,0.4875);
\draw[black!60, line width=0.35pt] (0.7500,1.2990) -- (0.0000,0.0000);
\draw[black!60, line width=0.35pt] (2.4000,2.2740) -- (3.1500,3.5731);
\draw[black!60, line width=0.35pt] (3.1500,3.5731) -- (3.0750,4.3846);
\draw[black!60, line width=0.35pt] (3.0750,4.3846) -- (2.3250,3.0856);
\fill[green!16] (2.2500,3.8971) -- (3.0750,4.3846) -- (3.0000,5.1962) -- cycle;
\draw[black!60, line width=0.35pt] (2.2500,3.8971) -- (3.0750,4.3846);
\draw[black!60, line width=0.35pt] (3.0750,4.3846) -- (3.0000,5.1962);
\draw[black!60, line width=0.35pt] (3.0000,5.1962) -- (2.2500,3.8971);
\draw[black!60, line width=0.35pt] (2.2500,3.8971) -- (1.5000,2.5981);
\fill[blue!14] (1.6500,0.9750) -- (2.4750,1.4625) -- (2.4000,2.2740) -- cycle;
\draw[black!60, line width=0.35pt] (1.6500,0.9750) -- (2.4750,1.4625);
\draw[black!60, line width=0.35pt] (2.4750,1.4625) -- (2.4000,2.2740);
\draw[black!60, line width=0.35pt] (2.4750,1.4625) -- (3.3000,1.9500);
\draw[black!60, line width=0.35pt] (3.3000,1.9500) -- (3.2250,2.7615);
\draw[black!60, line width=0.35pt] (3.2250,2.7615) -- (2.4000,2.2740);
\fill[orange!22] (2.4000,2.2740) -- (3.2250,2.7615) -- (3.1500,3.5731) -- cycle;
\draw[black!60, line width=0.35pt] (3.2250,2.7615) -- (3.1500,3.5731);
\node[red!70!black, font=\scriptsize] at (2.1000,2.3821) {$1$};
\node[green!45!black, font=\scriptsize] at (2.7750,4.4926) {$3$};
\node[blue!70!black, font=\scriptsize] at (2.1750,1.5705) {$2$};
\node[orange!85!black, font=\scriptsize] at (2.9250,2.8696) {$4$};
\fill[red!14] (5.2500,1.2990) -- (6.0000,2.5981) -- (4.5000,2.5981) -- cycle;
\draw[black!60, line width=0.35pt] (5.2500,1.2990) -- (6.0000,2.5981);
\draw[black!60, line width=0.35pt] (6.0000,2.5981) -- (4.5000,2.5981);
\draw[black!60, line width=0.35pt] (4.5000,2.5981) -- (5.2500,1.2990);
\draw[black!60, line width=0.35pt] (6.0000,0.0000) -- (6.7500,1.2990);
\draw[black!60, line width=0.35pt] (6.7500,1.2990) -- (6.0000,2.5981);
\draw[black!60, line width=0.35pt] (5.2500,1.2990) -- (6.0000,0.0000);
\draw[black!60, line width=0.35pt] (6.0000,2.5981) -- (6.7500,3.8971);
\draw[black!60, line width=0.35pt] (6.7500,3.8971) -- (5.2500,3.8971);
\draw[black!60, line width=0.35pt] (5.2500,3.8971) -- (4.5000,2.5981);
\draw[black!60, line width=0.35pt] (6.7500,3.8971) -- (7.5000,5.1962);
\draw[black!60, line width=0.35pt] (7.5000,5.1962) -- (6.0000,5.1962);
\draw[black!60, line width=0.35pt] (6.0000,5.1962) -- (5.2500,3.8971);
\fill[blue!14] (7.5000,2.5981) -- (8.2500,3.8971) -- (6.7500,3.8971) -- cycle;
\draw[black!60, line width=0.35pt] (7.5000,2.5981) -- (8.2500,3.8971);
\draw[black!60, line width=0.35pt] (8.2500,3.8971) -- (6.7500,3.8971);
\draw[black!60, line width=0.35pt] (6.7500,3.8971) -- (7.5000,2.5981);
\draw[black!60, line width=0.35pt] (8.2500,3.8971) -- (9.0000,5.1962);
\draw[black!60, line width=0.35pt] (9.0000,5.1962) -- (7.5000,5.1962);
\draw[black!60, line width=0.35pt] (6.7500,1.2990) -- (7.5000,2.5981);
\fill[green!16] (4.5000,2.5981) -- (5.2500,3.8971) -- (3.7500,3.8971) -- cycle;
\draw[black!60, line width=0.35pt] (5.2500,3.8971) -- (3.7500,3.8971);
\draw[black!60, line width=0.35pt] (3.7500,3.8971) -- (4.5000,2.5981);
\draw[black!60, line width=0.35pt] (5.2500,3.8971) -- (4.5000,5.1962);
\draw[black!60, line width=0.35pt] (4.5000,5.1962) -- (3.0000,5.1962);
\draw[black!60, line width=0.35pt] (3.0000,5.1962) -- (3.7500,3.8971);
\fill[orange!22] (5.2500,3.8971) -- (6.0000,5.1962) -- (4.5000,5.1962) -- cycle;
\draw[black!60, line width=0.35pt] (6.0000,5.1962) -- (4.5000,5.1962);
\node[red!70!black, font=\scriptsize] at (5.2500,2.1651) {$1$};
\node[blue!70!black, font=\scriptsize] at (7.5000,3.4641) {$2$};
\node[green!45!black, font=\scriptsize] at (4.5000,3.4641) {$3$};
\node[orange!85!black, font=\scriptsize] at (5.2500,4.7631) {$4$};
\draw[black!80, line width=0.9pt] (0.0000,0.0000) -- (6.0000,0.0000);
\draw[black!80, line width=0.9pt] (6.0000,0.0000) -- (3.0000,5.1962);
\draw[black!80, line width=0.9pt] (3.0000,5.1962) -- (0.0000,0.0000);
\draw[black!80, line width=0.9pt] (0.0000,0.0000) -- (3.3000,1.9500);
\draw[black!80, line width=0.9pt] (6.0000,0.0000) -- (3.3000,1.9500);
\draw[black!80, line width=0.9pt] (3.0000,5.1962) -- (3.3000,1.9500);
\draw[black!80, line width=0.9pt] (9.0000,5.1962) -- (6.0000,0.0000);
\draw[black!80, line width=0.9pt] (9.0000,5.1962) -- (3.0000,5.1962);
\node[orange!85!black, circle, fill=white, inner sep=0.3pt, font=\tiny] at (0.7500,0.0000) {$4$};
\node[blue!70!black, circle, fill=white, inner sep=0.3pt, font=\tiny] at (2.2500,0.0000) {$2$};
\node[green!45!black, circle, fill=white, inner sep=0.3pt, font=\tiny] at (3.7500,0.0000) {$3$};
\node[red!70!black, circle, fill=white, inner sep=0.3pt, font=\tiny] at (5.2500,0.0000) {$1$};
\node[orange!85!black, circle, fill=white, inner sep=0.3pt, font=\tiny] at (8.6250,4.5466) {$4$};
\node[blue!70!black, circle, fill=white, inner sep=0.3pt, font=\tiny] at (7.8750,3.2476) {$2$};
\node[green!45!black, circle, fill=white, inner sep=0.3pt, font=\tiny] at (7.1250,1.9486) {$3$};
\node[red!70!black, circle, fill=white, inner sep=0.3pt, font=\tiny] at (6.3750,0.6495) {$1$};
\node[green!45!black, circle, fill=white, inner sep=0.3pt, font=\tiny] at (0.4125,0.2437) {$3$};
\node[red!70!black, circle, fill=white, inner sep=0.3pt, font=\tiny] at (1.2375,0.7312) {$1$};
\node[blue!70!black, circle, fill=white, inner sep=0.3pt, font=\tiny] at (2.0625,1.2188) {$2$};
\node[orange!85!black, circle, fill=white, inner sep=0.3pt, font=\tiny] at (2.8875,1.7063) {$4$};
\node[green!45!black, circle, fill=white, inner sep=0.3pt, font=\tiny] at (5.6625,0.2437) {$3$};
\node[red!70!black, circle, fill=white, inner sep=0.3pt, font=\tiny] at (4.9875,0.7312) {$1$};
\node[blue!70!black, circle, fill=white, inner sep=0.3pt, font=\tiny] at (4.3125,1.2188) {$2$};
\node[orange!85!black, circle, fill=white, inner sep=0.3pt, font=\tiny] at (3.6375,1.7063) {$4$};
\node[green!45!black, circle, fill=white, inner sep=0.3pt, font=\tiny] at (3.0375,4.7904) {$3$};
\node[red!70!black, circle, fill=white, inner sep=0.3pt, font=\tiny] at (3.1125,3.9788) {$1$};
\node[orange!85!black, circle, fill=white, inner sep=0.3pt, font=\tiny] at (3.1875,3.1673) {$4$};
\node[blue!70!black, circle, fill=white, inner sep=0.3pt, font=\tiny] at (3.2625,2.3558) {$2$};
\node[blue!70!black, circle, fill=white, inner sep=0.3pt, font=\tiny] at (0.3750,0.6495) {$2$};
\node[red!70!black, circle, fill=white, inner sep=0.3pt, font=\tiny] at (1.1250,1.9486) {$1$};
\node[orange!85!black, circle, fill=white, inner sep=0.3pt, font=\tiny] at (1.8750,3.2476) {$4$};
\node[green!45!black, circle, fill=white, inner sep=0.3pt, font=\tiny] at (2.6250,4.5466) {$3$};
\node[blue!70!black, circle, fill=white, inner sep=0.3pt, font=\tiny] at (8.2500,5.1962) {$2$};
\node[red!70!black, circle, fill=white, inner sep=0.3pt, font=\tiny] at (6.7500,5.1962) {$1$};
\node[orange!85!black, circle, fill=white, inner sep=0.3pt, font=\tiny] at (5.2500,5.1962) {$4$};
\node[green!45!black, circle, fill=white, inner sep=0.3pt, font=\tiny] at (3.7500,5.1962) {$3$};
\node[blue!70!black, circle, fill=white, inner sep=0.3pt, font=\tiny] at (5.6250,0.6495) {$2$};
\node[red!70!black, circle, fill=white, inner sep=0.3pt, font=\tiny] at (4.8750,1.9486) {$1$};
\node[green!45!black, circle, fill=white, inner sep=0.3pt, font=\tiny] at (4.1250,3.2476) {$3$};
\node[orange!85!black, circle, fill=white, inner sep=0.3pt, font=\tiny] at (3.3750,4.5466) {$4$};
\node[below left=-1pt] at (0.0000,0.0000) {$1'$};
\node[below right=-1pt] at (6.0000,0.0000) {$2'$};
\node[above=0pt] at (3.0000,5.1962) {$3'$};
\node[circle, fill=white, inner sep=0.8pt, font=\small] at (3.3000,1.9500) {$4'$};
\node[above right=-1pt] at (9.0000,5.1962) {$1'$};
\draw[-{Stealth[length=5pt]}, gray!70, line width=0.8pt] (5.550,3.448) .. controls (5.650,2.748) and (5.050,2.248) .. (4.350,2.048);
\node[gray!70, font=\scriptsize, above right=-2pt] at (5.450,2.748) {fold};
\end{tikzpicture}

%% file: tikz/acyccounter4.tex
\begin{tikzpicture}[scale=1.05]
\fill[red!14] (0.0000,0.0000) -- (1.5000,0.0000) -- (0.8250,0.4875) -- cycle;
\draw[black!60, line width=0.35pt] (0.0000,0.0000) -- (1.5000,0.0000);
\draw[black!60, line width=0.35pt] (1.5000,0.0000) -- (0.8250,0.4875);
\draw[black!60, line width=0.35pt] (0.8250,0.4875) -- (0.0000,0.0000);
\draw[black!60, line width=0.35pt] (2.3250,0.4875) -- (3.1500,0.9750);
\draw[black!60, line width=0.35pt] (3.1500,0.9750) -- (2.4750,1.4625);
\draw[black!60, line width=0.35pt] (2.4750,1.4625) -- (1.6500,0.9750);
\draw[black!60, line width=0.35pt] (1.6500,0.9750) -- (2.3250,0.4875);
\draw[black!60, line width=0.35pt] (1.5000,0.0000) -- (2.3250,0.4875);
\draw[black!60, line width=0.35pt] (1.6500,0.9750) -- (0.8250,0.4875);
\draw[black!60, line width=0.35pt] (3.1500,0.9750) -- (4.6500,0.9750);
\draw[black!60, line width=0.35pt] (4.6500,0.9750) -- (3.9750,1.4625);
\draw[black!60, line width=0.35pt] (3.9750,1.4625) -- (2.4750,1.4625);
\fill[blue!14] (2.3250,0.4875) -- (3.8250,0.4875) -- (3.1500,0.9750) -- cycle;
\draw[black!60, line width=0.35pt] (2.3250,0.4875) -- (3.8250,0.4875);
\draw[black!60, line width=0.35pt] (3.8250,0.4875) -- (3.1500,0.9750);
\draw[black!60, line width=0.35pt] (1.5000,0.0000) -- (3.0000,0.0000);
\draw[black!60, line width=0.35pt] (3.0000,0.0000) -- (3.8250,0.4875);
\draw[black!60, line width=0.35pt] (3.8250,0.4875) -- (5.3250,0.4875);
\draw[black!60, line width=0.35pt] (5.3250,0.4875) -- (4.6500,0.9750);
\draw[black!60, line width=0.35pt] (3.0000,0.0000) -- (4.5000,0.0000);
\draw[black!60, line width=0.35pt] (4.5000,0.0000) -- (5.3250,0.4875);
\fill[orange!22] (4.5000,0.0000) -- (6.0000,0.0000) -- (5.3250,0.4875) -- cycle;
\draw[black!60, line width=0.35pt] (4.5000,0.0000) -- (6.0000,0.0000);
\draw[black!60, line width=0.35pt] (6.0000,0.0000) -- (5.3250,0.4875);
\fill[green!16] (2.4750,1.4625) -- (3.9750,1.4625) -- (3.3000,1.9500) -- cycle;
\draw[black!60, line width=0.35pt] (3.9750,1.4625) -- (3.3000,1.9500);
\draw[black!60, line width=0.35pt] (3.3000,1.9500) -- (2.4750,1.4625);
\node[red!70!black, font=\scriptsize] at (0.7750,0.1625) {$1$};
\node[blue!70!black, font=\scriptsize] at (3.1000,0.6500) {$2$};
\node[orange!85!black, font=\scriptsize] at (5.2750,0.1625) {$4$};
\node[green!45!black, font=\scriptsize] at (3.2500,1.6250) {$3$};
\fill[red!14] (3.8250,3.0856) -- (4.5000,2.5981) -- (3.7500,3.8971) -- cycle;
\draw[black!60, line width=0.35pt] (3.8250,3.0856) -- (4.5000,2.5981);
\draw[black!60, line width=0.35pt] (4.5000,2.5981) -- (3.7500,3.8971);
\draw[black!60, line width=0.35pt] (3.7500,3.8971) -- (3.8250,3.0856);
\draw[black!60, line width=0.35pt] (3.9750,1.4625) -- (4.6500,0.9750);
\draw[black!60, line width=0.35pt] (4.6500,0.9750) -- (4.5750,1.7865);
\draw[black!60, line width=0.35pt] (4.5750,1.7865) -- (3.9000,2.2740);
\draw[black!60, line width=0.35pt] (3.9000,2.2740) -- (3.9750,1.4625);
\draw[black!60, line width=0.35pt] (4.5750,1.7865) -- (4.5000,2.5981);
\draw[black!60, line width=0.35pt] (3.8250,3.0856) -- (3.9000,2.2740);
\fill[blue!14] (4.6500,0.9750) -- (5.3250,0.4875) -- (4.5750,1.7865) -- cycle;
\draw[black!60, line width=0.35pt] (4.6500,0.9750) -- (5.3250,0.4875);
\draw[black!60, line width=0.35pt] (5.3250,0.4875) -- (4.5750,1.7865);
\draw[black!60, line width=0.35pt] (5.3250,0.4875) -- (6.0000,0.0000);
\draw[black!60, line width=0.35pt] (6.0000,0.0000) -- (5.2500,1.2990);
\draw[black!60, line width=0.35pt] (5.2500,1.2990) -- (4.5750,1.7865);
\fill[orange!22] (4.5750,1.7865) -- (5.2500,1.2990) -- (4.5000,2.5981) -- cycle;
\draw[black!60, line width=0.35pt] (5.2500,1.2990) -- (4.5000,2.5981);
\draw[black!60, line width=0.35pt] (3.7500,3.8971) -- (3.0000,5.1962);
\draw[black!60, line width=0.35pt] (3.0000,5.1962) -- (3.0750,4.3846);
\draw[black!60, line width=0.35pt] (3.0750,4.3846) -- (3.8250,3.0856);
\draw[black!60, line width=0.35pt] (3.9000,2.2740) -- (3.1500,3.5731);
\draw[black!60, line width=0.35pt] (3.1500,3.5731) -- (3.2250,2.7615);
\draw[black!60, line width=0.35pt] (3.2250,2.7615) -- (3.9750,1.4625);
\draw[black!60, line width=0.35pt] (3.0750,4.3846) -- (3.1500,3.5731);
\fill[green!16] (3.3000,1.9500) -- (3.9750,1.4625) -- (3.2250,2.7615) -- cycle;
\draw[black!60, line width=0.35pt] (3.3000,1.9500) -- (3.9750,1.4625);
\draw[black!60, line width=0.35pt] (3.2250,2.7615) -- (3.3000,1.9500);
\node[red!70!black, font=\scriptsize] at (4.0250,3.1936) {$1$};
\node[blue!70!black, font=\scriptsize] at (4.8500,1.0830) {$2$};
\node[orange!85!black, font=\scriptsize] at (4.7750,1.8946) {$4$};
\node[green!45!black, font=\scriptsize] at (3.5000,2.0580) {$3$};
\fill[red!14] (0.0000,0.0000) -- (0.8250,0.4875) -- (0.7500,1.2990) -- cycle;
\draw[black!60, line width=0.35pt] (0.0000,0.0000) -- (0.8250,0.4875);
\draw[black!60, line width=0.35pt] (0.8250,0.4875) -- (0.7500,1.2990);
\draw[black!60, line width=0.35pt] (0.7500,1.2990) -- (0.0000,0.0000);
\draw[black!60, line width=0.35pt] (0.8250,0.4875) -- (1.5750,1.7865);
\draw[black!60, line width=0.35pt] (1.5750,1.7865) -- (1.5000,2.5981);
\draw[black!60, line width=0.35pt] (1.5000,2.5981) -- (0.7500,1.2990);
\draw[black!60, line width=0.35pt] (1.5750,1.7865) -- (2.3250,3.0856);
\draw[black!60, line width=0.35pt] (2.3250,3.0856) -- (2.2500,3.8971);
\draw[black!60, line width=0.35pt] (2.2500,3.8971) -- (1.5000,2.5981);
\draw[black!60, line width=0.35pt] (2.3250,3.0856) -- (3.0750,4.3846);
\draw[black!60, line width=0.35pt] (3.0750,4.3846) -- (3.0000,5.1962);
\draw[black!60, line width=0.35pt] (3.0000,5.1962) -- (2.2500,3.8971);
\fill[blue!14] (1.6500,0.9750) -- (2.4750,1.4625) -- (2.4000,2.2740) -- cycle;
\draw[black!60, line width=0.35pt] (1.6500,0.9750) -- (2.4750,1.4625);
\draw[black!60, line width=0.35pt] (2.4750,1.4625) -- (2.4000,2.2740);
\draw[black!60, line width=0.35pt] (2.4000,2.2740) -- (1.6500,0.9750);
\draw[black!60, line width=0.35pt] (0.8250,0.4875) -- (1.6500,0.9750);
\draw[black!60, line width=0.35pt] (2.4000,2.2740) -- (1.5750,1.7865);
\draw[black!60, line width=0.35pt] (2.4750,1.4625) -- (3.2250,2.7615);
\draw[black!60, line width=0.35pt] (3.2250,2.7615) -- (3.1500,3.5731);
\draw[black!60, line width=0.35pt] (3.1500,3.5731) -- (2.4000,2.2740);
\draw[black!60, line width=0.35pt] (3.1500,3.5731) -- (2.3250,3.0856);
\fill[orange!22] (2.3250,3.0856) -- (3.1500,3.5731) -- (3.0750,4.3846) -- cycle;
\draw[black!60, line width=0.35pt] (3.1500,3.5731) -- (3.0750,4.3846);
\fill[green!16] (2.4750,1.4625) -- (3.3000,1.9500) -- (3.2250,2.7615) -- cycle;
\draw[black!60, line width=0.35pt] (2.4750,1.4625) -- (3.3000,1.9500);
\draw[black!60, line width=0.35pt] (3.3000,1.9500) -- (3.2250,2.7615);
\node[red!70!black, font=\scriptsize] at (0.5250,0.5955) {$1$};
\node[blue!70!black, font=\scriptsize] at (2.1750,1.5705) {$2$};
\node[orange!85!black, font=\scriptsize] at (2.8500,3.6811) {$4$};
\node[green!45!black, font=\scriptsize] at (3.0000,2.0580) {$3$};
\fill[red!14] (8.2500,3.8971) -- (9.0000,5.1962) -- (7.5000,5.1962) -- cycle;
\draw[black!60, line width=0.35pt] (8.2500,3.8971) -- (9.0000,5.1962);
\draw[black!60, line width=0.35pt] (9.0000,5.1962) -- (7.5000,5.1962);
\draw[black!60, line width=0.35pt] (7.5000,5.1962) -- (8.2500,3.8971);
\draw[black!60, line width=0.35pt] (6.7500,3.8971) -- (8.2500,3.8971);
\draw[black!60, line width=0.35pt] (7.5000,5.1962) -- (6.0000,5.1962);
\draw[black!60, line width=0.35pt] (6.0000,5.1962) -- (6.7500,3.8971);
\draw[black!60, line width=0.35pt] (6.0000,2.5981) -- (6.7500,3.8971);
\draw[black!60, line width=0.35pt] (6.0000,5.1962) -- (5.2500,3.8971);
\draw[black!60, line width=0.35pt] (5.2500,3.8971) -- (6.0000,2.5981);
\draw[black!60, line width=0.35pt] (4.5000,2.5981) -- (6.0000,2.5981);
\draw[black!60, line width=0.35pt] (5.2500,3.8971) -- (3.7500,3.8971);
\draw[black!60, line width=0.35pt] (3.7500,3.8971) -- (4.5000,2.5981);
\fill[blue!14] (7.5000,2.5981) -- (8.2500,3.8971) -- (6.7500,3.8971) -- cycle;
\draw[black!60, line width=0.35pt] (7.5000,2.5981) -- (8.2500,3.8971);
\draw[black!60, line width=0.35pt] (6.7500,3.8971) -- (7.5000,2.5981);
\draw[black!60, line width=0.35pt] (6.7500,1.2990) -- (7.5000,2.5981);
\draw[black!60, line width=0.35pt] (6.0000,2.5981) -- (6.7500,1.2990);
\draw[black!60, line width=0.35pt] (6.0000,0.0000) -- (6.7500,1.2990);
\draw[black!60, line width=0.35pt] (6.0000,2.5981) -- (5.2500,1.2990);
\draw[black!60, line width=0.35pt] (5.2500,1.2990) -- (6.0000,0.0000);
\fill[orange!22] (5.2500,1.2990) -- (6.0000,2.5981) -- (4.5000,2.5981) -- cycle;
\draw[black!60, line width=0.35pt] (4.5000,2.5981) -- (5.2500,1.2990);
\fill[green!16] (5.2500,3.8971) -- (6.0000,5.1962) -- (4.5000,5.1962) -- cycle;
\draw[black!60, line width=0.35pt] (6.0000,5.1962) -- (4.5000,5.1962);
\draw[black!60, line width=0.35pt] (4.5000,5.1962) -- (5.2500,3.8971);
\draw[black!60, line width=0.35pt] (4.5000,5.1962) -- (3.0000,5.1962);
\draw[black!60, line width=0.35pt] (3.0000,5.1962) -- (3.7500,3.8971);
\node[red!70!black, font=\scriptsize] at (8.2500,4.7631) {$1$};
\node[blue!70!black, font=\scriptsize] at (7.5000,3.4641) {$2$};
\node[orange!85!black, font=\scriptsize] at (5.2500,2.1651) {$4$};
\node[green!45!black, font=\scriptsize] at (5.2500,4.7631) {$3$};
\draw[black!80, line width=0.9pt] (0.0000,0.0000) -- (6.0000,0.0000);
\draw[black!80, line width=0.9pt] (6.0000,0.0000) -- (3.0000,5.1962);
\draw[black!80, line width=0.9pt] (3.0000,5.1962) -- (0.0000,0.0000);
\draw[black!80, line width=0.9pt] (0.0000,0.0000) -- (3.3000,1.9500);
\draw[black!80, line width=0.9pt] (6.0000,0.0000) -- (3.3000,1.9500);
\draw[black!80, line width=0.9pt] (3.0000,5.1962) -- (3.3000,1.9500);
\draw[black!80, line width=0.9pt] (9.0000,5.1962) -- (6.0000,0.0000);
\draw[black!80, line width=0.9pt] (9.0000,5.1962) -- (3.0000,5.1962);
\node[red!70!black, circle, fill=white, inner sep=0.3pt, font=\tiny] at (0.7500,0.0000) {$1$};
\node[blue!70!black, circle, fill=white, inner sep=0.3pt, font=\tiny] at (2.2500,0.0000) {$2$};
\node[green!45!black, circle, fill=white, inner sep=0.3pt, font=\tiny] at (3.7500,0.0000) {$3$};
\node[orange!85!black, circle, fill=white, inner sep=0.3pt, font=\tiny] at (5.2500,0.0000) {$4$};
\node[red!70!black, circle, fill=white, inner sep=0.3pt, font=\tiny] at (8.6250,4.5466) {$1$};
\node[blue!70!black, circle, fill=white, inner sep=0.3pt, font=\tiny] at (7.8750,3.2476) {$2$};
\node[green!45!black, circle, fill=white, inner sep=0.3pt, font=\tiny] at (7.1250,1.9486) {$3$};
\node[orange!85!black, circle, fill=white, inner sep=0.3pt, font=\tiny] at (6.3750,0.6495) {$4$};
\node[red!70!black, circle, fill=white, inner sep=0.3pt, font=\tiny] at (0.4125,0.2437) {$1$};
\node[orange!85!black, circle, fill=white, inner sep=0.3pt, font=\tiny] at (1.2375,0.7312) {$4$};
\node[blue!70!black, circle, fill=white, inner sep=0.3pt, font=\tiny] at (2.0625,1.2188) {$2$};
\node[green!45!black, circle, fill=white, inner sep=0.3pt, font=\tiny] at (2.8875,1.7063) {$3$};
\node[orange!85!black, circle, fill=white, inner sep=0.3pt, font=\tiny] at (5.6625,0.2437) {$4$};
\node[blue!70!black, circle, fill=white, inner sep=0.3pt, font=\tiny] at (4.9875,0.7312) {$2$};
\node[red!70!black, circle, fill=white, inner sep=0.3pt, font=\tiny] at (4.3125,1.2188) {$1$};
\node[green!45!black, circle, fill=white, inner sep=0.3pt, font=\tiny] at (3.6375,1.7063) {$3$};
\node[red!70!black, circle, fill=white, inner sep=0.3pt, font=\tiny] at (3.0375,4.7904) {$1$};
\node[orange!85!black, circle, fill=white, inner sep=0.3pt, font=\tiny] at (3.1125,3.9788) {$4$};
\node[blue!70!black, circle, fill=white, inner sep=0.3pt, font=\tiny] at (3.1875,3.1673) {$2$};
\node[green!45!black, circle, fill=white, inner sep=0.3pt, font=\tiny] at (3.2625,2.3558) {$3$};
\node[red!70!black, circle, fill=white, inner sep=0.3pt, font=\tiny] at (0.3750,0.6495) {$1$};
\node[blue!70!black, circle, fill=white, inner sep=0.3pt, font=\tiny] at (1.1250,1.9486) {$2$};
\node[green!45!black, circle, fill=white, inner sep=0.3pt, font=\tiny] at (1.8750,3.2476) {$3$};
\node[orange!85!black, circle, fill=white, inner sep=0.3pt, font=\tiny] at (2.6250,4.5466) {$4$};
\node[red!70!black, circle, fill=white, inner sep=0.3pt, font=\tiny] at (8.2500,5.1962) {$1$};
\node[blue!70!black, circle, fill=white, inner sep=0.3pt, font=\tiny] at (6.7500,5.1962) {$2$};
\node[green!45!black, circle, fill=white, inner sep=0.3pt, font=\tiny] at (5.2500,5.1962) {$3$};
\node[orange!85!black, circle, fill=white, inner sep=0.3pt, font=\tiny] at (3.7500,5.1962) {$4$};
\node[blue!70!black, circle, fill=white, inner sep=0.3pt, font=\tiny] at (5.6250,0.6495) {$2$};
\node[orange!85!black, circle, fill=white, inner sep=0.3pt, font=\tiny] at (4.8750,1.9486) {$4$};
\node[red!70!black, circle, fill=white, inner sep=0.3pt, font=\tiny] at (4.1250,3.2476) {$1$};
\node[green!45!black, circle, fill=white, inner sep=0.3pt, font=\tiny] at (3.3750,4.5466) {$3$};
\node[below left=-1pt] at (0.0000,0.0000) {$1'$};
\node[below right=-1pt] at (6.0000,0.0000) {$2'$};
\node[above=0pt] at (3.0000,5.1962) {$3'$};
\node[circle, fill=white, inner sep=0.5pt, font=\small] at (3.3000,1.6700) {$4'$};
\node[above right=-1pt] at (9.0000,5.1962) {$1'$};
\draw[-{Stealth[length=5pt]}, gray!70, line width=0.8pt] (5.550,3.448) .. controls (5.650,2.748) and (5.050,2.248) .. (4.350,2.048);
\node[gray!70, font=\scriptsize, above right=-2pt] at (5.450,2.748) {fold};
\end{tikzpicture}

%% file: tikz/blockgraph_s3.tex
\begin{tikzpicture}[
  >={Stealth[length=2.2mm]},
  vtx/.style={draw,line width=0.8pt,rounded corners=2pt,inner sep=2.8pt,
              font=\small,fill=white},
  sink/.style={draw,line width=1.1pt,rounded corners=2pt,inner sep=3.2pt,
               font=\small\bfseries,fill=green!12,draw=green!50!black},
  bedge/.style={->,line width=0.8pt,black!72},
  elab/.style={font=\scriptsize,inner sep=1.5pt,fill=white,
               fill opacity=1,text opacity=1}]

\node[vtx]  (132) at (0,1.5)    {$132$};
\node[vtx]  (213) at (0,-1.5)   {$213$};
\node[vtx]  (312) at (2.7,1.5)  {$312$};
\node[vtx]  (321) at (2.7,-1.5) {$321$};
\node[sink] (231) at (5.4,0)    {$231$};

\draw[bedge] (132) to node[elab,pos=0.5]{$3$} (312);
\draw[bedge] (132) to node[elab,pos=0.62]{$2$} (231);
\draw[bedge] (213) to node[elab,pos=0.5]{$1$} (231);
\draw[bedge] (213) to node[elab,pos=0.72]{$2$} (312);
\draw[bedge] (312) to node[elab,pos=0.5]{$1$} (321);
\draw[bedge] (321) to node[elab,pos=0.5]{$3$} (231);
\end{tikzpicture}